\documentclass[11pt]{article}

\usepackage{amsmath,amssymb,amsthm}
\usepackage[utf8]{inputenc}
\usepackage{centernot}

\theoremstyle{plain}
\newtheorem{theorem}{Theorem}
\newtheorem{proposition}[theorem]{Proposition}

\newtheorem{corollary}[theorem]{Corollary}

\theoremstyle{definition}
\newtheorem{definition}[theorem]{Definition}
\newtheorem{remark}[theorem]{Remark}
\newtheorem{example}[theorem]{Example}
\newtheorem{question}[theorem]{Question}

\begin{document}
\title{Upper frequent hypercyclicity and related notions}
\author{Antonio Bonilla\thanks{The first author is supported by MEC and FEDER, Project MTM2013-47093-P}
\and
Karl-G. Grosse-Erdmann}
\date{}

\maketitle

\begin{abstract}
Enhancing a recent result of Bayart and Ruzsa we obtain a Birkhoff-type characterization of upper frequently hypercyclic operators and a corresponding Upper Frequent Hypercyclicity Criterion. As an application we characterize 
upper frequently hypercyclic weighted backward shifts on sequence spaces, which in turn allows us to come up with various counter-examples in linear dynamics that are substantially simpler than those previously obtained in the literature. More generally, we introduce the notion of upper Furstenberg families $\mathcal{A}$ and show that our main results hold for $\mathcal{A}$-hypercyclic operators with respect to such families.
\end{abstract}

\section{Introduction}

Baire category arguments play a crucial role in topological dynamics, and in particular in linear dynamics. One of their main applications, the powerful Birkhoff transitivity theorem, lies behind several important results. On the other hand, a recent and by now central notion in linear dynamics, that of frequently hypercyclic operators as introduced by Bayart and Grivaux \cite{BaGr04}, \cite{BaGr06}, has to be studied without the use of Baire category. 
Indeed, since the set of frequently hypercyclic vectors is always meagre there cannot be a Birkhoff-type theorem for frequent hypercyclicity.

Now, in \cite{Shk09}, Shkarin has introduced the weaker notion of upper (or $\mathcal{U}$-)frequently hypercyclic operators. In a recent paper, Bayart and Ruzsa \cite{BaRu15} have shown that the set of upper frequently hypercyclic vectors is residual whenever it is non-empty. It is therefore natural to seek a Birkhoff-type enhancement of the Bayart-Ruzsa result, which is the starting point of our present investigation.

The main aim of this paper is to show that there is indeed a Birkhoff transitivity theorem for upper frequent hypercyclicity, and to deduce from it an Upper Frequent Hypercyclicity Criterion. As an application we obtain a characterization of upper frequently hypercyclic weighted shifts on $c_0$ (and on quite general sequence spaces) that simplifies considerably the one given by Bayart and Ruzsa \cite{BaRu15}. This will then allow us to come up with some counter-examples that are substantially simpler than those previously obtained in the literature.

One may wonder why ordinary hypercyclicity and upper frequent hypercyclicity behave so differently from frequent hypercyclicity. We therefore pursue our investigation within the general framework of $\mathcal{A}$-hypercyclicity, where $\mathcal{A}$ is a Furstenberg family. This also allows us to cover the notion of reiterative hypercyclicity that was introduced by Peris \cite{Per05} and that has recently been studied by B\`es, Menet, Peris and Puig \cite{BMPP15}, \cite{Men15}.

The paper is organized as follows. In Section \ref{s-fur} we consider Furstenberg families and the corresponding $\mathcal{A}$-hypercyclicity. We introduce, in particular, the notion of an upper Furstenberg family and consider various examples. Section \ref{s-Bir} presents a Birkhoff transitivity theorem for upper Furstenberg families. As an application we show that the inverse of an invertible reiteratively hypercyclic operator is itself reiteratively hypercyclic. Another consequence, an $\mathcal{A}$-Hypercyclicity Criterion for upper Furstenberg families, is deduced in Section \ref{s-AHC}. With its help we can characterize $\mathcal{A}$-hypercyclic weighted backward shifts for upper Furstenberg families, see Section \ref{s-weiupp}. Motivated by this result we will then characterize, in Section \ref{s-weiarb}, $\mathcal{A}$-hypercyclic weighted backward shifts for arbitrary Furstenberg families; the proof is necessarily constructive. In the final Section \ref{s-ex} we then construct four counter-examples, three of which represent significant simplifications of previous constructions in the literature.

We end this introduction by recalling the three central notions of this paper; see Examples \ref{ex-fur1} and \ref{ex-fur} for the densities involved. 

\begin{definition}\label{d-uppreit}
Let $T:X\to X$ be a continuous mapping on a topological space $X$. 

(a) $T$ is called \textit{frequently hypercyclic} if there exists a point $x\in X$ such that, for any non-empty open set $U$ in $X$,
\[
\underline{\text{dens}}\; \{n\geq 0 : T^nx\in U\}>0.
\]
The point $x$ is then called \textit{frequently hypercyclic} for $T$, and the set of such points is denoted by $FHC(T)$.

(b) $T$ is called \textit{upper frequently hypercyclic} if there exists a point $x\in X$ such that, for any non-empty open set $U$ in $X$,
\[
\overline{\text{dens}}\; \{n\geq 0 : T^nx\in U\}>0.
\]
The point $x$ is then called \textit{upper frequently hypercyclic} for $T$, and the set of such points is denoted by $UFHC(T)$.

(c) $T$ is called \textit{reiteratively hypercyclic} if there exists a point $x\in X$ such that, for any non-empty open set $U$ in $X$,
\[
\overline{\text{Bd}}\; \{n\geq 0 : T^nx\in U\}>0.
\]
The point $x$ is then called \textit{reiteratively hypercyclic} for $T$, and the set of such points is denoted by $RHC(T)$.
\end{definition} 

Let us also recall that a mapping $T:X\to X$ is called \textit{hypercyclic} if some point $x\in X$ (also called \textit{hypercyclic}) has a dense orbit $\{T^nx : n\geq 0\}$; the set of such points is denoted by $HC(T)$. If $X$ is a complete metric space then $T$ is called \textit{chaotic} if it is hypercyclic and if it admits a dense set of periodic points. Note that, in this paper, we take the liberty of using the term `hypercyclic' even in a non-linear setting. Moreover, we will admit arbitrary sequences $(T_n)$ instead of only the iterates $(T^n)$ of a single mapping $T$.

For an introduction to linear dynamics we refer to the recent monographs \cite{BaMa09} and \cite{GrPe11}.

\section{Furstenberg families and $\mathcal{A}$-hypercyclicity}\label{s-fur}

We adopt the following terminology.

\begin{definition} 
A non-empty family $\mathcal{A}$ of subsets of $\mathbb{N}_0$ is called a \textit{Furstenberg family} if it is \textit{hereditary upward}, that is, if
\[
A\in\mathcal{A}, B\supset A \Longrightarrow B\in \mathcal{A}.
\]
\end{definition}

\begin{definition}
Let $\mathcal{A}$ be a Furstenberg family, and let $T_n:X\to Y$, $n\geq 0$, and $T:X\to X$ be continuous mappings, where $X$ and $Y$ are topological spaces.

(a) $(T_n)$ is called \textit{$\mathcal{A}$-universal} if there exists a point $x\in X$ such that, for any non-empty open set $U$ in $Y$,
\[
\{ n\geq 0 : T_nx\in U\}\in \mathcal{A}.
\]
The point $x$ is then called \textit{$\mathcal{A}$-universal} for $(T_n)_n$.

(b) $T$ is called \textit{$\mathcal{A}$-hypercyclic} if there exists a point $x\in X$ such that, for any non-empty open set $U$ in $X$,
\[
\{ n\geq 0 : T^nx\in U\}\in \mathcal{A}.
\]
The point $x$ is then called \textit{$\mathcal{A}$-hypercyclic} for $T$, and the set of such points is denoted by
\[
\mathcal{A}HC(T).
\]
\end{definition} 

\begin{remark} A word needs to be said about our terminology. The idea of considering families of sets of integers to quantify dynamical properties can be traced back at least to Gottschalk and Hedlund \cite[Section 3]{GoHe55}, where they appear as collections of `admissible sets' in connection with recurrence phenomena. They then played a prominent role in Furstenberg \cite{Fur81}; for example, $\mathcal{A}$-recurrence is considered in Chapter 9. This motivated Akin \cite{Aki97} to introduce the notion of `Furstenberg families'; he studies notions like $\mathcal{A}$-recurrence and $\mathcal{A}$-transitivity in great detail. Furstenberg families, sometimes also simply called families, have recently been studied intensively in non-linear dynamics, see for example Glasner \cite{Gla04}, Xiong, L\"u and Tan \cite{XLT07}, Li \cite{Li11}, Huang, Li and Ye \cite{HLY12}, and Chen, Li and L\"u \cite{CLL15}. In linear dynamics, the notion of $\mathcal{A}$-hypercyclicity is due to Shkarin \cite[Section 5]{Shk09}; a detailed study was first undertaken by B\`es, Menet, Peris and Puig \cite{BMPP15}. 

We adopt here the usual terminology from linear dynamics to speak of universality in the case of general sequences $(T_n)$ and of hypercyclicity in the case of iterates of a single mapping, see \cite{BaMa09} and \cite{GrPe11}. In contrast, in Li \cite{Li11}, $\mathcal{A}$-hypercyclic points appear as $\mathcal{A}$-transitive, $\mathcal{A}$-hypercyclic mappings as $\mathcal{A}$-point transitive.  
\end{remark} 

The three most important notions in linear dynamics are those of hypercyclicity, frequent hypercyclicity and chaos. Two of those are particular instances of $\mathcal{A}$-hypercyclicities. Recently, Grivaux and Matheron \cite{GrMa14} introduced a new type of $\mathcal{A}$-hypercyclicity that is even stronger than frequent hypercyclicity.

\begin{example}\label{ex-fur1}
(a) Let $\mathcal{A}_{\neq\varnothing}$ be the Furstenberg family of non-empty subsets of $\mathbb{N}_0$. Then $\mathcal{A}_{\neq\varnothing}$-hypercyclicity is ordinary hypercyclicity.

(b) Let $\mathcal{A}_{\text{ld}}$ be the Furstenberg family of sets $A$ of positive lower density, that is, with
\[
\underline{\text{dens}}\; A = \liminf_{N\to\infty} \frac{1}{N+1}\text{card} (A \cap [0,N])>0.
\]
Then $\mathcal{A}_{\text{ld}}$-hypercyclicity is frequent hypercyclicity. 

(c) Hindman \cite{Hin90} has studied the Furstenberg family $\mathcal{A}_{\text{Hin}}$ of sets $A$ for which
\[
\lim_{N\to\infty} \underline{\text{dens}}\; \bigcup_{n=0}^N (A-n) = 1,
\]
where $A-n=\{k-n : k\in A, k\geq n\}$. He showed that any set for which the lower density is positive and coincides with its upper Banach density (see Example \ref{ex-fur}(b) below) has this property. For want of a better word we will call $\mathcal{A}_{\text{Hin}}$-hypercyclic operators \textit{very frequently hypercyclic}. Thus $x\in X$ is very frequently hypercyclic for $T:X\to X$ if and only if, for any non-empty open set $U$ in $X$,
\[
\underline{\text{dens}}\; \Big\{ k\geq 0 : T^k x \in \bigcup_{n=0}^N T^{-n}(U)\Big\} \to 1\quad\text{as $N\to\infty$}.
\]
Clearly, very frequent hypercyclicity implies frequent hypercyclicity. In the proof of \cite[Proposition 2.9]{GrMa14}, Grivaux and Matheron show that any continuous mapping on a second countable topological space that admits an ergodic Borel measure of full support is very frequently hypercyclic; in particular, any operator on a separable Fr\'echet space that satisfies the Frequent Hypercyclicity Criterion is very frequently hypercyclic, see Murillo-Arcila and Peris \cite{MuPe13}. This provides a rich source of such operators, see \cite{BaMa09}, \cite{GrPe11}. 
\end{example}

In this context, the following problem seems to be open.

\begin{question}\label{q-chaos}
Is there a Furstenberg family $\mathcal{A}$ for which $\mathcal{A}$-hypercyclicity is equivalent to chaos (for operators on Banach spaces, say)? It would already be of interest to have a Furstenberg family $\mathcal{A}$ for which $\mathcal{A}$-hypercyclicity implies chaos.
\end{question}

In Akin \cite{Aki97} one finds an in-depth study of Furstenberg families, along with useful properties that may be imposed on them. The following is a concept that excludes the trivial family of all subsets of $\mathbb{N}_0$.

\begin{definition} 
A Furstenberg family is called \textit{proper} if it does not contain the empty set (equivalently, if it does not coincide with $\mathcal{P}(\mathbb{N}_0)$). 
\end{definition} 

Then we have the following.

\begin{proposition}\label{p-afhc}
Let $\mathcal{A}$ be a proper Furstenberg family, and let $T_n:X\to Y$, $n\geq 0$, be continuous mappings between topological spaces.

\emph{(a)} If $x\in X$ is an $\mathcal{A}$-universal point for $(T_n)$ then $\{T_nx:n\geq 0\}$ is dense in $Y$.

\emph{(b)} If $X=Y$ and the mappings $T_n$, $n\geq 0$, commute and have dense range, then the set of $\mathcal{A}$-universal points for $(T_n)$ is either empty or dense.
\end{proposition}

\begin{proof} Assertion (a) is obvious. For (b), let $x\in X$ be an $\mathcal{A}$-universal point. In view of (a) it suffices to show that, for any $m\geq 0$, $T_mx$ is also $\mathcal{A}$-universal. Thus fix $m\geq 0$ and a non-empty open set $U\subset X$. Then $T_m^{-1}(U)$ is a non-empty open set, and hence
\[
\{n\geq 0 : T_n(T_mx)\in U\} = \{n\geq 0 : T_nx\in T_m^{-1}(U)\} \in \mathcal{A}.
\]
\end{proof}

Our main objective will be a generalization of the Birkhoff transitivity theorem to $\mathcal{A}$-hypercyclicity. In particular we want to deduce that the set $\mathcal{A}HC(T)$ is residual. However, for frequent hypercyclicity, the set of frequently hypercyclic vectors is always meagre, as was recently shown by Moothathu \cite{Moo13}, Bayart and Ruzsa \cite{BaRu15}, and Grivaux and Matheron \cite{GrMa14}. Thus we need to impose some restrictions on $\mathcal{A}$.

\begin{definition} 
(a) A Furstenberg family $\mathcal{A}$ is called \textit{upper} if it is proper and it can be written as
\[
\mathcal{A} = \bigcup_{\delta\in D} \mathcal{A}_{\delta}\quad \mbox{with}\quad \mathcal{A}_{\delta}:=\bigcap_{\mu\in M} \mathcal{A}_{\delta,\mu}
\]
for some families $\mathcal{A}_{\delta,\mu}$ ($\delta\in D,\mu\in M$), where $D$ is arbitrary but $M$ is countable, and such that
\begin{itemize}
\item[(i)] each family $\mathcal{A}_{\delta,\mu}$ is \textit{finitely hereditary upward}, that is, for any $A\in \mathcal{A}_{\delta,\mu}$ there is a finite set $F\subset\mathbb{N}_0$ such that
\[
B\supset A\cap F  \Longrightarrow B\in\mathcal{A}_{\delta,\mu};
\]
\item[(ii)] $\mathcal{A}$ is \textit{uniformly left-invariant}, that is, for any $A\in \mathcal{A}$ there is some $\delta\in D$ such that, for all $n\geq 0$,
\[
A-n\in \mathcal{A}_\delta.
\]
\end{itemize}
\end{definition}

For some proofs we will need an additional property.

\begin{definition} 
(a) A Furstenberg family $\mathcal{A}$ is called \textit{finitely invariant} (\textit{f.i.}) if, for any $A\in \mathcal{A}$ and all $n\geq 0$,
\[
A\setminus [0,n] \in \mathcal{A}.
\]

(b) Let $\mathcal{A} = \bigcup_{\delta\in D}\mathcal{A}_{\delta}$ with $\mathcal{A}_{\delta}=\bigcap_{\mu\in M} \mathcal{A}_{\delta,\mu}$ be an upper Furstenberg family. Then it is called \textit{uniformly finitely invariant} (\textit{u.f.i.}) if, for any $A\in \mathcal{A}$, there is some $\delta\in D$ such that, for all $n\geq 0$,
\[
A\setminus [0,n] \in \mathcal{A}_\delta.
\]
\end{definition}

\begin{remark}  
(a) Note that, for uniform left-invariance and uniform finite invariance, we demand not only that $A-n$ and $A\setminus [0,n]$ belong to $\mathcal{A}$ but that they should all belong to the same $\mathcal{A}_\delta$. Moreover, both properties are an immediate consequence of the following stronger property: for any $A \in \mathcal{A}_\delta$ and any $n\geq 0$, $A-n\in \mathcal{A}_\delta$ and $A+n\in \mathcal{A}_\delta$. 

(b) It is easy to see that if $\mathcal{A}$ is an u.f.i.~upper Furstenberg family and $A\in \mathcal{A}$, then there is some $\delta\in D$ such that $A\setminus F \in \mathcal{A}_\delta$ for any finite set $F$. A similar remark applies to f.i.~Furstenberg families.
\end{remark} 

We provide here an extensive list of examples of upper Furstenberg families, many of which can be found in Grekos \cite{Gre05}. 

\begin{example}\label{ex-fur}
(a) Our foremost example is the family $\mathcal{A}_{\text{ud}}$ of sets $A$ of positive upper density, that is, with
\[
\overline{\text{dens}}\; A = \limsup_{N\to\infty} \frac{1}{N+1}\text{card} (A \cap [0,N])>0.
\]
The corresponding $\mathcal{A}_{\text{ud}}$-hypercyclicity is upper frequent hypercyclicity, see Definition \ref{d-uppreit}. We have that $\mathcal{A}_{\text{ud}}= \bigcup_{\delta>0}\bigcap_{n\geq 0} \mathcal{A}_{\delta,n}$, where $A\in \mathcal{A}_{\delta,n}$ if and only if 
\begin{equation}\label{eq-fhu}
\exists N\geq n:\;\frac{1}{N+1}\text{card}(A \cap [0,N])>\delta.
\end{equation}
Then each family $\mathcal{A}_{\delta,n}$ is finitely hereditary upward because if $A$ satisfies \eqref{eq-fhu} for some $N\geq n$ then $B\supset A\cap [0,N]$ implies that $B\in \mathcal{A}_{\delta,n}$. Uniform left-invariance follows from the fact that $\overline{\text{dens}}\; (A-n)=  \overline{\text{dens}}\; A$ for all $n\geq 0$. Hence the family of sets of positive upper density is an upper Furstenberg family.

Incidentally, the related P\'olya maximum density 
\[
\lim_{\alpha\to 1-}\limsup_{N\to\infty} \frac{\text{card}(A \cap [0,N])-\text{card}(A \cap [0,\alpha N])}{(1-\alpha)N}
\]
plays an important r\^ole in complex and harmonic analysis, see \cite[p. 559]{Pol29}, \cite{Koo98}. Since this density is at least as large as upper density (see \cite[Satz III]{Pol29}) it is easily seen that the families of sets of positive density coincide for these two notions.

(b) A second major example is provided by upper Banach density, which is defined as
\begin{align*}
\overline{\text{Bd}}\; A &= \sup\Big\{\delta\geq 0 : \forall n\geq 0\;\exists N\geq n, m\geq 0 :\frac{1}{N+1}\text{card}(A \cap [m,m+N])\geq \delta\Big\},
\end{align*}
see \cite{Fur81}, \cite{Hin90}, \cite{GTT10}. The terminology was apparently motivated by the early works of Banach (see \cite{Hin90}) but an equivalent variant was also defined by P\'olya \cite[p.~561]{Pol29}.

Upper Banach density coincides with what was originally called upper uniform density, that is,
\begin{align*}
\overline{\text{Bd}}\; A &= \lim_{N\to\infty} \limsup_{m\to\infty}\frac{1}{N+1}\text{card}(A \cap [m,m+N]),
\end{align*}
as introduced by Brown and Freedman \cite{BrFr90}, see also \cite{FrSe81}, \cite{BrFr87}. In fact, a more convenient form for us is that 
\begin{align*}
\overline{\text{Bd}}\; A &=  \inf_{N\geq 0} \sup_{m\geq 0}\frac{1}{N+1}\text{card}(A \cap [m,m+N]).
\end{align*}
The proof of the equality of these three (and further) values can be found in \cite{GTT10}, see also \cite{SaTo03}, \cite{Fre11}.

Let $\mathcal{A}_{\text{uBd}}$ be the family of sets of positive upper Banach density. The corresponding $\mathcal{A}_{\text{uBd}}$-hypercyclicity is reiterative hypercyclicity, see Definition \ref{d-uppreit}. Now, by the third representation of upper Banach density above we have that $\mathcal{A}_{\text{uBd}}= \bigcup_{\delta>0}\bigcap_{N\geq 0} \mathcal{A}_{\delta,N}$, where $A\in \mathcal{A}_{\delta,N}$ if and only if 
\[
\exists m\geq 0:\;\frac{1}{N+1}\text{card}(A \cap [m,m+N])>\delta.
\]
It follows as in (a) that the family of sets of positive upper Banach density is an upper Furstenberg family.

(c) One may generalize the upper density from example (a) in an obvious way. Let $W=(w_{n,k})_{n,k\geq 0}$ be an infinite matrix of non-negative numbers with the property
 that, for all $n\geq 0$, $w_{n,k}$ is decreasing in $k$, and that, for all $k\geq 0$, $w_{n,k}\to 0$ as $n\to\infty$. The corresponding upper $W$-density is defined as
\[
W\text{-}\overline{\text{dens}}\; A = \limsup_{N\to\infty} \sum_{k\in A} w_{N,k}.
\] 
Then the family $\mathcal{A}_{\text{u}W}$ of sets of positive upper $W$-density is an upper Furstenberg family. Indeed, $\mathcal{A}_{\text{u}W} = \bigcup_{\delta >0} \bigcap_{n\geq 0} \mathcal{A}_{\delta,n}$, where $A\in \mathcal{A}_{\delta,n}$ if and only if 
\[
\exists N\geq n, m\geq 0:\;\sum_{\substack{k\in A\\k\leq m}} w_{N,k}>\delta,
\]
which implies the finite hereditary upward property. It is also easy to see that 
$W\text{-}\overline{\text{dens}}\; (A-n)\geq  W\text{-}\overline{\text{dens}}\; A$ for all $n\geq 0$, which implies the uniform left-invariance property. In particular, when $W$ is the Ces\`aro-matrix we get the upper density from example (a). Other interesting choices of $W$ are given by matrices arising from summability theory; see Freedman and Sember \cite{FrSe81}.

(d) Let $w=(w_n)_{n\geq 0}$ be a non-negative decreasing sequence such that $w_0>0$ and $W_N:=\sum_{k=0}^N w_k\to\infty$ as $N\to\infty$. Then the corresponding upper weighted density is defined as
\[
w\text{-}\overline{\text{dens}}\; A = \limsup_{N\to\infty} \frac{1}{W_N}\sum_{\substack{k\leq N\\k\in A}} w_k.
\]
This is the special case of the upper $W$-density in (c) for $W=(\frac{w_k}{W_n})_{n\geq 0, k\leq n}$. Thus the family of sets of positive upper $w$-density is an upper Furstenberg family. A popular choice is given by the upper $\alpha$-density, where $w_n=\frac{1}{(n+1)^\alpha}$, $\alpha \leq 1$; in particular, for $\alpha=0$ we have the upper density of (a), for $\alpha=1$ the so-called upper logarithmic density.

(e) Upper exponential density is defined by
\[
\text{exp-}\overline{\text{dens}}\; A = \limsup_{N\to\infty}\frac{1}{\log (N+1)}\log^+(\text{card}(A \cap [0,N])).
\]
Again it is easily seen that the family of sets of positive upper exponential density is an upper Furstenberg family.

(f) Instead of positive density in the previous examples one may even require maximal density. Thus, let $\mathcal{A}_{\text{mud}}$ be the family of sets $A$ for which
\[
\overline{\text{dens}}\; A = 1.
\]
Then $\mathcal{A}_{\text{mud}}=\bigcap_{m,n\geq 1} \mathcal{A}_{(m,n)}$, where $A\in \mathcal{A}_{(m,n)}$ if and only if
\[
\exists N\geq n: \frac{1}{N+1}\text{card}(A \cap [0,N])>1-\frac{1}{m};
\]
one may take the set $D$ in the definition of upper Furstenberg families as an arbitrary singleton. It is shown as in (a) that $\mathcal{A}_{\text{mud}}$ is an upper Furstenberg family. In a completely analogous way, the families of sets of upper Banach density 1, of upper weighted density 1 (under the assumptions of (d)) and of upper exponential density 1 are upper Furstenberg families.

(g) A variant of the families in (d) was suggested by Shkarin \cite[Section 5]{Shk09}. Let $\varphi=(\varphi_k)_{k\geq 0}$ be a decreasing sequence of positive numbers with $\sum_{k=0}^\infty \varphi_k=\infty$ and define $\mathcal{A}_\varphi$ as the family of subsets $A$ of $\mathbb{N}_0$ such that
\[
\sum_{k\in A}{\varphi_k}=\infty.
\]
Then $\mathcal{A}_\varphi=\bigcap_{m> 0} \mathcal{A}_{m}$, where $A\in \mathcal{A}_{m}$ if and only if
\[
\exists N\geq 0: \sum_{\substack{k\in A\\k\leq N}} \varphi_k>m;
\]
one may take $D$ again as a singleton. One sees easily that $\mathcal{A}_\varphi$ is an upper Furstenberg family.

(h) We finally consider the family $\mathcal{A}_\infty$ of infinite subsets. This is the special case of the families considered in (g) when $\varphi_k=1$ for all $k\geq 0$. We therefore have an upper Furstenberg family. The corresponding $\mathcal{A}_\infty$-hypercyclicity is ordinary hypercyclicity provided the underlying space  is Hausdorff and has no isolated points (see \cite[Proposition 1.15]{GrPe11}). Note that the Furstenberg family $\mathcal{A}_{\neq \varnothing}$ of Example \ref{ex-fur1}(a), which also defines ordinary hypercyclicity, is not upper because it cannot satisfy condition (ii). 
\end{example}

We note that each of the upper Furstenberg families discussed above is u.f.i.

\begin{remark} For the Furstenberg family of sets of upper density 1, see (f), $\mathcal{A}_{\text{mud}}$-universal series were introduced by Papachristodoulos \cite{Pap13} and subsequently studied in \cite{KNP12} and \cite{MoMu15}. However, if $X$ is a Hausdorff topological space with at least two points and if $\mathcal{A}$ is one of the Furstenberg families of sets of upper density 1, of upper Banach density 1 or of upper weighted density 1 (under the assumptions stated in (d)) then no mapping on $X$ can be $\mathcal{A}$-hypercyclic. This can be proved easily by the argument in B\`es, Menet, Peris and Puig \cite[Proposition 3]{BMPP15}.
\end{remark}

In contrast to the last remark, the fact that $\log(\delta(N+1))/\log(N+1)\to 1$ as $N\to \infty$ shows the following.

\begin{proposition}
Any upper frequently hypercyclic mapping is $\mathcal{A}$-hypercyclic for the Furstenberg family of sets of upper exponential density $1$. 
\end{proposition}

\section{A Birkhoff theorem for upper Furstenberg families}\label{s-Bir}

We have the following analogue of the Birkhoff transitivity theorem in our general context.

\begin{theorem}\label{t-BirAHC}
Let $X$ be a complete metric space, $Y$ a separable metric space, and $T_n:X\to Y$, $n\geq 0$, continuous mappings. Let $\mathcal{A} = \bigcup_{\delta\in D}\bigcap_{\mu\in M} \mathcal{A}_{\delta,\mu}$ be an upper Furstenberg family. 
Consider the following:
\begin{itemize}
\item[\rm (a)] for any non-empty open subset $V$ of $Y$ there is some $\delta\in D$ such that for any non-empty open subset $U$ of $X$ there is some $x\in U$ such that
\[
\{n\geq 0 : T_n x\in V\} \in \mathcal{A}_\delta;
\]
\item[\rm (b)]  for any non-empty open subset $V$ of $Y$ there is some $\delta\in D$ such that for any non-empty open subset $U$ of $X$ and any $\mu\in M$ there is some $x\in U$ such that
\[
\{n\geq 0 : T_n x\in V\} \in \mathcal{A}_{\delta,\mu};
\]
\item[\rm (c)] the set of $\mathcal{A}$-universal points for $(T_n)$ is residual in $X$;
\item[\rm (d)] $(T_n)$ admits an $\mathcal{A}$-universal point.
\end{itemize}
Then \emph{(a)} $\Longrightarrow$ \emph{(b)} $\Longrightarrow$ \emph{(c)} $\Longrightarrow$ \emph{(d)}.
 
Moreover, if $X=Y$ and $T_n=T^n$, $n\geq 0$, for some continuous mapping $T:X\to X$, then the four assertions are equivalent.
\end{theorem}

\begin{proof}
(a) $\Longrightarrow$ (b). This is nothing more than a permutation of quantifiers from ``$\exists x\;\forall\mu$'' to ``$\forall\mu\;\exists x$.''

(b) $\Longrightarrow$ (c). Let $(V_k)_{k\geq 1}$ be a countable base of non-empty open sets in $Y$. For $k\geq 1$, let $\delta_k\in D$ be the parameter associated to $V_k$ by (b). We consider the set
\begin{equation}\label{eq-E}
E:= \bigcap_{k\geq 1,\mu\in M} O_{k,\mu},
\end{equation}
where 
\[
O_{k,\mu} = \{ x\in X : \{n\geq 0 : T_n x\in V_k\} \in \mathcal{A}_{\delta_k,\mu}\}.
\] 
Then each set $O_{k,\mu}$ is open. Indeed, let $A:=\{n\geq 0 : T_n x\in V_k\} \in \mathcal{A}_{\delta_k,\mu}$. Let $F\subset\mathbb{N}_0$ be a finite set associated to $A$ by the finite hereditary upward property. By continuity of $T$ there is a neighbourhood $W$ of $x$ such that $T_n y\in V_k$ for any $y\in W$ and $n\in A\cap F$. Consequently,
\[
\{n\geq 0 : T_n y\in V_k\}\in \mathcal{A}_{\delta_k,\mu}
\]
for any $y\in W$, so that $O_{k,\mu}$ is an open set.

On the other hand, condition (b) tells us that each set $O_{k,\mu}$ is dense. It follows from the Baire category theorem that $E$ is a dense $G_\delta$-set. Now, clearly, every point in $E$ is $\mathcal{A}$-universal, which implies (c).

(c) $\Longrightarrow$ (d) is trivial.

Now suppose that $X=Y$ and that $T:X\to X$ is a continuous mapping. Suppose that (d) holds, and let $y$ be an $\mathcal{A}$-universal point for $T$. Let $V$ be a non-empty open set in $X$. By definition,
\[
\{n\geq 0 : T^n y \in V\} \in \mathcal{A},
\]
which implies by uniform left-invariance that there is some $\delta\in D$ such that, for all $m\geq 0$,
\[
\{n\geq 0 : T^n y \in V\}-m \in \mathcal{A}_{\delta}.
\]
Let $U$ be another non-empty open set. Since $\mathcal{A}$ is proper, Proposition \ref{p-afhc}(a) implies that there is some $m\geq 0$ such that $x:=T^my\in U$. Thus,
\[
\{n\geq 0 : T^nx\in V\} = \{k\geq 0 : T^ky\in V\}-m \in \mathcal{A}_{\delta},
\]
which shows (a).
\end{proof}

\begin{remark} 
As usual, the result holds, more generally, on any Baire space $X$ and any second countable space $Y$.
\end{remark}

Another look at the preceding proof shows that if the set $D$ is a singleton then the set $E$ in \eqref{eq-E} is actually the set of all $\mathcal{A}$-universal points. Thus we have the following.

\begin{corollary}\label{c-gdelta}
Let $T:X\to X$ be a continuous mapping on a separable complete metric space $X$, and let $\mathcal{A}$ be an upper Furstenberg family that has a representation with ${ \text{\rm card}}(D)=1$. Then the set of $\mathcal{A}$-hypercyclic points for $T$ is either empty or a dense $G_\delta$-set.
\end{corollary}

This result applies, for example, to the $\mathcal{A}$-hypercyclicities with respect to the Furstenberg families of upper exponential density 1 or for $\mathcal{A}_\varphi$, which includes ordinary hypercyclicity. We do not know, however, if the set $UFHC(T)$ of upper frequently hypercyclic points is always a $G_\delta$-set.

We spell out the two particular cases of Theorem \ref{t-BirAHC} that are of greatest interest to us.

\begin{corollary}\label{c-BirUFHC}
Let $T:X\to X$ be a continuous mapping on a separable complete metric space $X$. Then the following assertions are equivalent:
\begin{itemize}
\item[\rm (a)] for any non-empty open subset $V$ of $X$ there is some $\delta >0$ such that for any non-empty open subset $U$ of $X$ there is some $x\in U$ such that
\[
\overline{\text{\rm dens}}\;\{n\geq 0 : T^n x\in V\} >\delta;
\]
\item[\rm (b)] for any non-empty open subset $V$ of $X$ there is some $\delta>0$ such that for any non-empty open subset $U$ of $X$ and any $n\geq 0$ there are some $x\in U$ and $N\geq n$ such that
\[
\frac{1}{N+1}\text{\rm card}\{n\leq N : T^n x\in V\}>\delta;
\]
\item[\rm (c)] $T$ is upper frequently hypercyclic.
\end{itemize}
In that case the set $UFHC(T)$ of upper frequently hypercyclic points for $T$ is residual.
\end{corollary}

This result improves Proposition 21 of Bayart and Ruzsa \cite{BaRu15}, which had been the starting point of our investigation.

For the second special case one has to note that once a mapping $T$ is reiteratively hypercyclic then every hypercyclic point for $T$ is even reiteratively hypercyclic (see \cite[Theorem 14]{BMPP15}), so that $RHC(T)=HC(T)$, which then is a dense $G_\delta$-set, see \cite[Theorem 1.16, Exercise 1.2.3]{GrPe11}. 

\begin{corollary}\label{c-BirRHC}
Let $T:X\to X$ be a continuous mapping on a separable complete metric space $X$. Then the following assertions are equivalent:
\begin{itemize}
\item[\rm (a)] for any non-empty open subset $V$ of $X$ there is some $\delta >0$ such that for any non-empty open subset $U$ of $X$ there is some $x\in U$ such that
\[
\overline{\text{\rm Bd}}\;\{n\geq 0 : T^n x\in V\} >\delta;
\]
\item[\rm (b)] for any non-empty open subset $V$ of $X$ there is some $\delta>0$ such that for any non-empty open subset $U$ of $X$ and any $N\geq 0$ there are some $x\in U$ and $m\geq 0$ such that
\[
\frac{1}{N+1}\text{\rm card}\{n \in [m, m+N] : T^n x\in V\}>\delta;
\]
\item[\rm (c)] $T$ is reiteratively hypercyclic.
\end{itemize}
In that case the set $RHC(T)=HC(T)$ of reiteratively hypercyclic points for $T$ is a dense $G_\delta$-set.
\end{corollary}

The corollary easily implies Menet's result \cite{Men15} that every chaotic mapping on a complete metric space is reiteratively hypercyclic; but the proof is essentially the same as Menet's direct proof.

Here is another interesting consequence of Corollary \ref{c-BirRHC}. The classical Birkhoff transitivity theorem implies immediately that the inverse of an invertible hypercyclic mapping is also hypercyclic. The same property is then enjoyed by chaotic operators. However, it is an open problem whether the analogous statement holds for frequently hypercyclic operators, see \cite{BaGr06}, and the question also seems to be open for upper frequent hypercyclicity. We can give here a positive answer for reiterative hypercyclicity.

\begin{theorem}
Let $T$ be an invertible mapping on a complete metric space $X$. If $T$ is reiteratively hypercyclic then so is $T^{-1}$.
\end{theorem}

\begin{proof} We prove that $T^{-1}$ satisfies condition (b) of Corollary \ref{c-BirRHC}. For this, we fix a reiteratively hypercyclic point $y$ for $T$. Then let $V\subset X$ be a non-empty open set. By definition,
\[
\delta:=\overline{\text{Bd}}\{n\geq 0 : T^n y \in V\}>0.
\]
Let $U\subset X$ be a non-empty open set. Then for any $N\geq 0$ there is some $m\geq 0$ such that
\[
\frac{1}{N+1}\text{\rm card}\{n \in [m, m+N] : T^n y\in V\}>\frac{\delta}{2}.
\]
By reiterative hypercyclicity of $y$ there is then also some $k\geq 0$ such that $x:=T^{m+N+k}y\in U$. Going from this point $x$ backwards via $T^{-1}$ one obtains that
\[
\frac{1}{N+1}\text{\rm card}\{n \in [k, k+N] : (T^{-1})^n x\in V\} >\frac{\delta}{2},
\]
which had to be shown.
\end{proof}

The property of the family of sets of upper Banach density used in this proof seems to be quite specific to this family and the family $\mathcal{A}_\infty$, so that we do not have similar results for the other upper Furstenberg families discussed in Example \ref{ex-fur}.

\begin{remark} In a similar way one may deduce from Corollary \ref{c-BirUFHC} the result of Bayart and Ruzsa \cite{BaRu15} that if an invertible mapping $T$ on a complete metric space is frequently hypercyclic then its inverse $T^{-1}$ is upper frequently hypercyclic.
\end{remark}

\section{An $\mathcal{A}$-Hypercyclicity Criterion for upper Furstenberg families}\label{s-AHC}

From now on we will restrict our attention to the linear setting, that is, we consider sequences $(T_n)_n$ of  (continuous and linear) operators between topological vector spaces. 

The topology of a metrizable topological vector space may always be defined by an F-norm, see \cite[pp. 2--5]{KPR84}, \cite[Definition 2.9]{GrPe11}, which will be denoted indiscriminately by $\|\cdot\|$. An F-space is a metrizable topological vector space that is complete under some F-norm, and hence under any equivalent F-norm. 

For classical hypercyclicity, the Birkhoff transitivity theorem leads easily to the celebrated Hypercyclicity Criterion. In the same vein we have the following.

\begin{theorem}[$\mathcal{A}$-Hypercyclicity Criterion]\label{t-ahc}
Let $X$ be an F-space, $Y$ a separable metrizable topological vector space, and $T_n: X \rightarrow Y$, $n\geq 0$, operators. Let $\mathcal{A}=\bigcup_{\delta\in D} \mathcal{A}_\delta$ be an upper Furstenberg family. 

Suppose that there are a dense subset $X_0$ of $X$, a dense subset $Y_0$ of $Y$ and mappings $S_n:Y_0\rightarrow X$, $n\geq 0$, with the following property: For any $y\in Y_0$ and $\varepsilon>0$ there exist $A\in\mathcal{A}$ and $\delta\in D$ such that
\begin{itemize}
\item[\rm (i)] for any $x\in X_0$ there is some $B\in \mathcal{A}_\delta$, $B\subset A$, such that, for any $n\in B$,
\[
\|T_nx\|<\varepsilon;
\]
\item[\rm (ii)] 
\[
\sum _{n\in A} S_n y\; \text{converges};
\]
\item[\rm (iii)] for any $m\in A$,
\[
\|T_m \sum _{n\in A} S_n y-y\Big\| <\varepsilon.
\]
\end{itemize}
Then $(T_n)$ is $\mathcal{A}$-universal with a residual set of $\mathcal{A}$-universal vectors.
\end{theorem}

\begin{remark}\label{r-ahc}
It seems more natural to demand, instead of condition (iii), the following three conditions: for any $m\in A$ we have that 
\begin{itemize}
\item[\rm (iiia)] 
\[
\Big\|T_m \sum _{\substack{n\in A\\n<m}} S_n y\Big\| <\varepsilon;
\]
\item[\rm (iiib)] 
\[
\Big\|T_m \sum _{\substack{n\in A\\n>m}} S_n y\Big\| <\varepsilon;
\]
\item[\rm (iiic)]
\[
\|T_mS_my -y\| <\varepsilon;
\]
\end{itemize}
the criterion still works, of course, but is (a priori) a weaker result.
\end{remark}

\begin{proof} It suffices to verify condition (a) of Theorem  \ref{t-BirAHC}. Thus, let $V$ be a non-empty open subset of $Y$. We can then choose $y_0\in V\cap Y_0$; let $\varepsilon>0$ be such that $y\in V$ whenever $\|y-y_0\|<2\varepsilon$. By assumption, there exist some $A\in\mathcal{A}$ and $\delta\in D$ such that (i)--(iii) hold. In particular,
\[
z:=\sum _{n\in A} S_n y_0
\]
converges in $X$ and, for any $m\in A$,
\begin{equation}\label{eq-AHC}
\|T_m z - y_0\| <{\varepsilon}.
\end{equation}

Now let $U$ be a non-empty open subset of $X$. By density of $X_0$ we can find some $x_0\in X_0$ such that
\[
x_0 \in U -z.
\]
By condition (i) there is then some $B\in \mathcal{A}_\delta$, $B\subset A$, such that, for any $m\in B$,
\begin{equation}\label{eq-AHC2}
\|T_mx_0\|<{\varepsilon}.
\end{equation}
Hence
\[
x:=x_0+z\in U,
\]
and it follows from \eqref{eq-AHC} and \eqref{eq-AHC2} that, for any $m\in B$,
\[
\|T_mx -y_0\|<2\varepsilon
\]
and hence $T_mx\in V$. Thus we have shown that $\{ n\geq 0 : T_nx\in V\}\in \mathcal{A}_\delta$.
\end{proof}

In particular we have the following.

\begin{corollary}[Upper Frequent Hypercyclicity Criterion]\label{c-ufhc}
Let $T: X \rightarrow X$ be an operator on a separable F-space $X$.

Suppose that there are dense subsets $X_0, Y_0$ of $X$ and mappings $S_n:Y_0\rightarrow X$, $n\geq 0$, with the following property: For any $y\in Y_0$ and $\varepsilon>0$ there exists $A\subset\mathbb{N}_0$ with $\overline{\text{\rm dens}}\;A>0$ and $\delta>0$ such that
\begin{itemize}
\item[\rm (i)] for any $x\in X_0$ there is some $B\subset A$ with $\overline{\text{\rm dens}}\;B>\delta$ such that, for any $n\in B$,
\[
\|T^nx\|<\varepsilon;
\]
\item[\rm (ii)] 
\[
\sum _{n\in A} S_n y\; \text{converges};
\]
\item[\rm (iii)] for any $m\in A$,
\[
\Big\|T^m \sum _{n\in A} S_n y-y\Big\| <\varepsilon.
\]
\end{itemize}
Then $T$ is upper frequently hypercyclic.
\end{corollary}

In the same way we have a Reiterative Hypercyclicity Criterion when we replace upper density by upper Banach density.

\begin{remark}\label{r-compAHC} 
B\`es, Menet, Peris and Puig \cite{BMPP15} also state an $\mathcal{A}$-Hypercyclicity Criterion, and it is important to understand the difference between their criterion and ours. First of all, our criterion requires an upper Furstenberg family, while theirs works for very general such families; it covers, in particular, also the case of frequent hypercyclicity. However, not even for upper Furstenberg families the two criteria are comparable. The essential difference is that the criterion of B\`es et al. requires a sequence $(A_k)$ of sets in $\mathcal{A}$ while ours works with a single set $A$. On the other hand, our criterion requires that $(T_n)$ is small on a dense subset of $X$, for suitably many $n$ (see condition (i)), which does not appear in B\`es et al. These differences will have consequences later on, see Remark \ref{r-compChar} below.
\end{remark}

Still, Theorem \ref{t-BirAHC} allows us to deduce, for many upper Furstenberg families, the criterion of B\`es et al. under somewhat weaker hypotheses. 

\begin{theorem}[$\mathcal{A}$-Hypercyclicity Criterion, second version]\label{t-ahc2}
Let $X$ be an F-space, $Y$ a separable metrizable topological vector space, and $T_n: X \rightarrow Y$, $n\geq 0$, operators. Let $\mathcal{A}$ be an u.f.i.~upper Furstenberg family. 

Suppose that there are a dense subset $Y_0$ of $Y$, mappings $S_n:Y_0\rightarrow X$, $n\geq 0$, and sets $A_k\in\mathcal{A}$, $k\geq 1$, such that for any $y\in Y_0$ we have the following:
\begin{itemize}
\item[\rm (i)] for any $k\geq 1$,  
\[
\sum_{n\in A_k} S_n y\; \text{converges};
\]
\item[\rm (ii)] for any $k_0\geq 1$ and any $\varepsilon>0$ there exists $k\geq k_0$ such that for any finite set $F\subset A_k$ and any $m\in \bigcup_{l\geq 1} A_l,$ $m\notin F$, we have that
\[
\Big\|T_m\sum_{n\in F}S_ny\Big\|<\varepsilon,
\]
and such that, for any $\delta>0$, there exists $l_0\geq 1$ such that for any finite set $F\subset A_k$ and any $m\in \bigcup_{l\geq l_0} A_l,$ $m\notin F$, we have that
\[
\Big\|T_m\sum_{n\in F}S_ny\Big\|<\delta;
\]
\item[\rm (iii)]
\[
\sup_{m\in A_k}\|T_m S_m y-y\| \to 0\quad\text{as $k\to\infty$.}
\]
\end{itemize}
Then $(T_n)$ admits an $\mathcal{A}$-universal vector.
\end{theorem}

\begin{proof} Let $(y_l)_{l\geq 1}$ be a dense sequence in $Y_0$. As in the proof in \cite{BMPP15}, conditions (ii) and (iii) imply that there exists a subsequence $(B_k):=(A_{n_k})$ of $(A_k)$ such that, for all $l\geq 1$,
\begin{itemize}
\item[\rm (I)] for any finite set $F\subset B_l$ and any $m\in \bigcup_{k\geq 1} B_k,$ $m\notin F$, 
\[
\Big\|T_m\sum_{n\in F}S_ny_l\Big\|<\frac{1}{2^l};
\]
\item[\rm (II)]
for any $j<l$, any finite set $F\subset B_j$ and any $m\in B_l,$ $m\notin F$, 
\[
\Big\|T_m\sum_{n\in F}S_ny_j\Big\|<\frac{1}{l2^l};
\]
\item[\rm (III)] for any $m\in B_l,$  
\[
\|T_m S_m y_l-y_l\| <\frac{1}{2^l}.
\]
\end{itemize}
Now let $\widetilde{X}$ be the closure of the set
\begin{equation}\label{eq-x0}
X_0 := \Big\{ \sum_{j=1}^J \sum_{n\in F_j} S_n y_j : J\geq 1, F_j\subset B_j \text{ finite}\Big\}
\end{equation}
in $X$, where the $F_j$ may be empty. We will apply Theorem \ref{t-BirAHC} to the mappings $T_n$ restricted to $\widetilde{X}$, which is a complete metric space.

Thus, let $V$ be a non-empty open subset of $Y$. Then there exists some $l\geq 1$ such that the open ball around $y_l$ of radius $\frac{4}{2^l}$ is contained in $V$. Since $B_l\in \mathcal{A}$ there is by uniform finite invariance some $\delta>0$ such that $B_l\cap [N,\infty)\in \mathcal{A}_\delta$ for all $N\geq 0$.

Now let $U$ be a non-empty open subset of $\widetilde{X}$. By definition of this space there are finitely many finite sets $F_j\subset B_j,$ $1\leq j\leq J$, such that
\[
\sum_{j=1}^J \sum_{n\in F_j} S_n y_j \in U.
\]
Then, for $N$ sufficiently large,
\[
x:= \sum_{j=1}^J \sum_{n\in F_j} S_n y_j + \sum_{n\in B_l, n\geq N} S_n y_l \in U,
\]
where we have used condition (i); note that $x$ belongs to $\widetilde{X}$ as a limit of elements in $X_0$. We may assume that $J\geq l$ and that $N>\max_{1\leq j\leq J}\max F_j$. Then we have for any $m\in B_l$, $m\geq N$,
\begin{align*}
T_mx &- y_l =\\
&T_m \sum_{j=1}^{l-1} \sum_{n\in F_j} S_n y_j + T_m \sum_{j=l+1}^J \sum_{n\in F_j} S_n y_j 
+ \lim_{M\to\infty}\Big(T_m \sum_{n\in C_{l,M,m}} S_n y_l\Big) + T_mS_my_l -y_l,
\end{align*}
where $C_{l,M,m} = F_l \cup (B_l\cap [N,M])\setminus\{m\}$. Now, by (I), (II) and (III),
\[
\|T_m x- y_l\| \leq  \sum_{j=1}^{l-1} \frac{1}{l2^l} + \sum_{j=l+1}^J \frac{1}{2^j} + \frac{1}{2^l} + \frac{1}{2^l}<\frac{4}{2^l},
\]
so that $T_mx\in V$ for all $m\in B_l$, $m\geq N$. Since $B_l\cap [N,\infty)\in \mathcal{A}_\delta$, we have that condition (a) of Theorem \ref{t-BirAHC} is satisfied. It follows that $(T_n)$ admits an $\mathcal{A}$-universal point in $\widetilde{X}$, hence in $X$.
\end{proof}

We note that the proof also shows that the set of $\mathcal{A}$-universal points for $(T_n)$ is residual in $X$ provided that the set $X_0$, defined in \eqref{eq-x0}, is dense in $X$.

\section{Weighted shift operators and upper Furstenberg families}\label{s-weiupp}

Weighted shifts play an eminent role in linear dynamics. We will show in this section that the $\mathcal{A}$-Hypercyclicity Criterion lends itself easily to a characterization of the $\mathcal{A}$-hypercyclicity of weighted shifts. In some cases we obtain conditions that are much simpler than those previously established.

We adopt here the terminology of \cite[Section 4.1]{GrPe11}. In particular, an F-sequence space  is an F-space that is a subspace of the space $\mathbb{K}^{\mathbb{N}_0}$ of all (real or complex) sequences and such that each coordinate functional $x=(x_n)_{n\geq 0}\to x_k$, $k\geq 0$, is continuous. Fr\'echet sequence spaces are defined analogously. The canonical unit sequences are denoted by $e_n=(\delta_{n,k})_{k\geq 0}$. 

Let $w=(w_n)_{n\geq 1}$ be a weight sequence, that is, a sequence of non-zero scalars. The corresponding weighted backward shift $B_w$ is defined by $B_w((x_n)_{n\geq 0})=(w_{n+1}x_{n+1})_{n\geq 0}$. The closed graph theorem implies that whenever $B_w$ maps an F-sequence space $X$ into itself then $B_w$ is an operator on $X$. The unweighted shift (with $w_n=1$ for all $n\geq 0$) is denoted by $B$. It is by now well known that by conjugacy the dynamical properties of $B_w$ on $X$ coincide with those of $B$ on a related sequence space, see below.

Thus we start by studying the unweighted shift. As usual, $\|\cdot\|$ will denote an F-norm defining the topology of an F-space.

\begin{theorem}\label{t-bahc}
Let $X$ be an F-sequence space in which $(e_n)_{n\geq 0}$ is a basis. Suppose that the backward shift $B$ is an operator on $X$. Let $\mathcal{A}$ be a u.f.i.~upper Furstenberg family.

\emph{(a)} If, for any $p\geq 0$ and $\varepsilon>0$, there exists some $A\in\mathcal{A}$, $A\subset [p,\infty)$, such that
\begin{itemize}
\item[\rm (i)] 
\[
\sum_{n\in A} e_{n} \quad \text{converges in $X$};
\]
\item[\rm (ii)] for any $m\in A$,
\[
\Big\|\sum _{\substack{n\in A\\n>m}} e_{n-m+p}\Big\| <\varepsilon;
\]
\end{itemize}
then $B$ is $\mathcal{A}$-hypercyclic. 

\emph{(b)} If $X$ is a Fr\'echet sequence space in which $(e_n)_{n\geq 0}$ is an unconditional basis then the converse also holds.
\end{theorem}

We note that, by continuity of $B$ and the fact that $m\geq p$ for all $m\in A$, the convergence of the series in (i) implies the convergence of the series in (ii).

\begin{proof} (a). We first show that we may assume $A$ to have arbitrarily large gaps: fixing an arbitrary $N\geq 1$ one may have that
\[
\text{for any $n,m\in A$ with $n>m$}, (n-m) > N.
\]
Indeed, by the continuity of the coordinate functionals there is some $\varepsilon_N>0$ such that $|x_0|,\ldots,|x_{N+p}|<1$ whenever $\|x\|<\varepsilon_N$. Choosing then $\varepsilon=\varepsilon_N>0$, (ii) implies that $n-m+p>N+p$ for any $n,m\in A$, $n>m$. 

We now let $S_n=F^n$, $n\geq 0$, where $F$ denotes the forward shift that maps $(x_0,x_1,\ldots)$ to $(0,x_0,x_1,\ldots)$. We choose $X_0=Y_0$ as the set of finite sequences, which is dense by assumption. It suffices then to show that the conditions of Theorem \ref{t-ahc} hold for $T_n=B^n$, $n\geq 0$. 

To this end, let 
\[
y=(y_0,\ldots,y_p,0,\ldots)\in Y_0
\]
with $p\geq 0$, and let $\varepsilon>0$. There is some $\eta>0$ such that if $x^{(0)},x^{(1)},\ldots,x^{(p)}\in X$ have norm less than $\eta$ then 
\begin{equation}\label{eq-eps}
\Big\|\sum_{j=0}^p y_j B^{p-j}x^{(j)}\Big\|<\varepsilon.
\end{equation}
There then exists a set $\widetilde{A}\in\mathcal{A}$, $\widetilde{A}\subset [p,\infty)$ such that assumptions (i) and (ii) hold for $\varepsilon$ replaced by $\eta$. By our argument above we can assume that $\widetilde{m}-\widetilde{n}>p$ whenever $\widetilde{m},\widetilde{n}\in \widetilde{A}$, $\widetilde{m}>\widetilde{n}$. We set 
\[
A=\widetilde{A}-p,
\]
and the left-translation invariance of $\mathcal{A}$ assures us that $A\in\mathcal{A}$. We can now verify the conditions of the $\mathcal{A}$-Hypercyclicity Criterion.  

First, since $B^nx\to 0$ as $n\to\infty$ for all $x\in X_0$ and since $\mathcal{A}$ is u.f.i., condition (i) of Theorem \ref{t-ahc} holds.

Next, by assumption (i) we have that
\[
\sum_{n\in A} e_{n+p} = \sum_{\widetilde{n}\in \widetilde{A}} e_{\widetilde{n}}
\]
converges. Consequently,  
\[
\sum _{n\in A} S_n y = \sum _{n\in A} \sum_{j=0}^p y_j e_{j+n} =
\sum_{j=0}^p y_j B^{p-j}\sum _{n\in A} e_{n+p} 
\]
converges. Thus condition (ii) of Theorem \ref{t-ahc} holds.

Finally, we have for $m\in A$, 
\[
B^{m} \sum _{n\in A} S_n y -y = \sum _{\substack{n\in A\\n<m}} B^{m-n} y + B^mF^my + \sum _{\substack{n\in A\\n>m}} F^{n-m}y-y=\sum _{\substack{n\in A\\n>m}} F^{n-m}y,
\]
where we have used that $m-n>p$ whenever $m,n\in A$, $n<m$, and the fact that $F$ is the right inverse of $B$. Now, writing $m=\widetilde{m}-p$ with $\widetilde{m}\in\widetilde{A}$,
\[
\sum _{\substack{n\in A\\n>m}} F^{n-m} y = \sum _{\substack{\widetilde{n}\in \widetilde{A}\\\widetilde{n}>\widetilde{m}}} F^{\widetilde{n}-\widetilde{m}} y = \sum _{\substack{\widetilde{n}\in \widetilde{A}\\\widetilde{n}>\widetilde{m}}} \sum_{j=0}^p y_j e_{j+\widetilde{n}-\widetilde{m}} = 
\sum_{j=0}^p y_j B^{p-j}\sum _{\substack{\widetilde{n}\in \widetilde{A}\\\widetilde{n}>\widetilde{m}}}  e_{\widetilde{n}-\widetilde{m}+p}, 
\]
so that \eqref{eq-eps} and assumption (ii) for $\widetilde{A}$ and $\eta$ imply condition (iii) of Theorem \ref{t-ahc}. This completes the proof of (a).

(b) Now suppose that $X$ is a Fr\'echet sequence space and that $(e_n)$ is an unconditional basis in $X$. We will need the following property of such a basis, see \cite[3.3.9]{KaGu81} with \cite[Theorem 2.17]{Rud73}:
\begin{itemize}
\item[(M)] If $(x_n)\in X$ and $(a_n)$ is a bounded sequence of scalars then $(a_nx_n)\in X$. Moreover, for any $\varepsilon>0$ there is some $\delta>0$ such that, if $\|x\|\leq \delta$ and $\sup_n|a_n|\leq 2$, then  $\|(a_nx_n)\|<
\varepsilon$. 
\end{itemize}
In addition, the continuity of the coordinate functionals implies the following:
\begin{itemize}
\item[(F)] For any $p\geq 0$ there is some $\alpha_p>0$ such that, whenever $x=(x_n)\in X$ and $\|x\|<\alpha_p$, then $|x_j|\leq \frac{1}{2}$ for $j=0,\ldots,p$. 
\end{itemize}
Now let $x\in X$ be an $\mathcal{A}$-hypercyclic vector for $B$. Fix $p\geq 0$ and $\varepsilon>0$. Let $\delta>0$ and $\alpha_p$ be numbers given by (M) and (F), and set
\[
y=\sum_{j=0}^p e_j.
\]
Then, by $\mathcal{A}$-hypercyclicity of $x$, there is some set $A\in\mathcal{A}$ such that, for all $n\in A$, 
\[
\|B^nx-y\|<\min (\delta, \alpha_p).
\]
We can assume by finite invariance that $A\subset [p,\infty)$. Now, (F) implies that, for any $n\in A$ and $j=0,\ldots,p$,
\[
|x_{n+j}-1|\leq \frac{1}{2},
\]
hence 
\begin{equation}\label{eq-Bhyp2}
\Big|\frac{1}{x_{n+j}}\Big|\leq 2. 
\end{equation}

On the one hand, by \eqref{eq-Bhyp2} with $j=0$ and property (M), the fact that $x\in X$ implies that
\[
\sum_{n\in A} \frac{1}{x_n}x_n e_n
\]
converges in $X$. This shows condition (i). 

On the other hand, since $B^mx-y\in X$ with $\|B^mx-y\|<\delta$ for all $m\in A$, \eqref{eq-Bhyp2} with $j=p$ and property (M) imply that, for all $m\in A$,
\[
\sum_{\substack{n\in A\\n>m}} \frac{1}{x_{n+p}}[B^mx-y]_{n-m+p} e_{n-m+p} 
\]
converges in $X$ and has F-norm less than $\varepsilon$; here, $[z]_k$ denotes the $k$th entry of the sequence $z$. But since $n-m+p>p$ we have that
\[
[B^mx-y]_{n-m+p}=x_{n+p}.
\]
This shows that, for all $m\in A$,
\[
\sum_{\substack{n\in A\\n>m}} e_{n-m+p} 
\]
converges in $X$ and has F-norm less than $\varepsilon$, which shows condition (ii).
\end{proof}

\begin{remark}\label{r-equiv}
Since $B^{p-j}e_{n-m+p}=e_{n-m+j}$ we see from the continuity of $B$ that one may replace condition (ii) equivalently by the stronger condition
\begin{itemize}
\item[\rm (ii')] for any $m\in A$ and any $0\leq j\leq p$
\[
\Big\|\sum _{\substack{n\in A\\n>m}} e_{n-m+j}\Big\| <\varepsilon.
\]
\end{itemize}
\end{remark}

One can transfer the result from the unweighted shift $B$ to arbitrary weighted shift operators $B_w$ via conjugacy, see \cite{MaPe02}. Indeed, let $B_w$ be an operator on an F-sequence space $X$. Defining $v_n=(\prod_{\nu=1}^n w_\nu)^{-1}$, $n\geq 0$, and $X(v)=\{(x_n) : (x_nv_n)\in X\}$ one easily sees that $B_w$ is conjugate to $B$ via the bijection $\phi_v: X(v)\to X$, $(x_n)\to (x_nv_n)$; see \cite[Section 4.1]{GrPe11} for details. Now, $\mathcal{A}$-hypercyclicity is preserved under conjugacies, which follows directly from the fact that for any non-empty open subset $V$ of $X(v)$ and $x\in X$, $\{ n\geq 0 : B_w^n\phi_v(x) \in V\}= \{ n\geq 0 : B^n(x) \in \phi_v^{-1}(V)\}$; compare also with \cite[Proposition 1.19]{GrPe11}.

In this way we obtain the following.

\begin{theorem}\label{t-bwahc}
Let $X$ be an F-sequence space in which $(e_n)_{n\geq 0}$ is a basis. Suppose that the weighted backward shift $B_w$ is an operator on $X$. Let $\mathcal{A}$ be a u.f.i.~upper Furstenberg family.

\emph{(a)} If, for any $p\geq 0$ and $\varepsilon>0$, there exists some $A\in\mathcal{A}$, $A\subset [p,\infty)$, such that
\begin{itemize}
\item[\rm (i)] 
\[
\sum_{n\in A} \frac{1}{\prod_{\nu=1}^n w_\nu}e_{n} \quad \text{converges in $X$};
\]
\item[\rm (ii)] for any $m\in A$,
\[
\Big\|\sum _{\substack{n\in A\\n>m}} \frac{1}{\prod_{\nu=1}^{n-m+p} w_\nu}e_{n-m+p}\Big\| <\varepsilon;
\]
\end{itemize}
then $B_w$ is $\mathcal{A}$-hypercyclic.

\emph{(b)} If $X$ is a Fr\'echet sequence space in which $(e_n)_{n\geq 0}$ is an unconditional basis then the converse also holds.
\end{theorem}

Since our focus in the next section will be on shift operators on the space $c_0=c_0(\mathbb{N}_0)$ of null sequences, we state this special case explicitly.

\begin{corollary}\label{c-bwahc}
Let $w=(w_n)_{n\geq 1}$ be a bounded sequence of non-zero scalars and $\mathcal{A}$ a u.f.i.~upper Furstenberg family. Then the weighted backward shift $B_w$ is $\mathcal{A}$-hypercyclic on $c_0$ if and only if, for any $p\geq 0$ and $M>0$, there exists some $A\in\mathcal{A}$, $A\subset [p,\infty)$, such that
\begin{itemize}
\item[\rm (i)] 
\[
|w_1 w_2\cdots w_n| \to\infty\quad\text{as $n\to\infty$, $n\in A$};
\]
\item[\rm (ii)] for any $n,m\in A$, $n>m$,
\[
|w_1 w_2\cdots w_{n-m+p}| > M.
\]
\end{itemize}
\end{corollary}

\begin{remark}\label{r-compChar}
Let us discuss these results in the light of what is known in the literature. Upper frequently hypercyclic weighted shifts on $c_0$ were recently characterized by Bayart and Ruzsa \cite[Theorem 6]{BaRu15}. Our characterization -- take for $\mathcal{A}$ the family of sets of positive upper density -- is substantially simpler. While Bayart and Ruzsa need to work with an interacting sequence $(A_p)$ of sets, our condition only involves a single set $A$.

The same comment applies to Theorem 6 of B\`es, Menet, Peris and Puig \cite{BMPP15}. They characterize $\mathcal{A}$-hypercyclicity of weighted shifts on $c_0$ and $\ell_p$ $(1\leq p<\infty)$ for Furstenberg families $\mathcal{A}$ under very minor restrictions on $\mathcal{A}$. Again, their conditions involve sets $A_p\in\mathcal{A}$ that interact. Thus, in the special case of u.f.i.~upper Furstenberg families, our conditions are considerably weaker. This applies, for example, for reiterative hypercyclicity, as studied by these authors.

On the other hand, when we consider Furstenberg families that are not upper, or when we pass from sequence spaces over $\mathbb{N}_0$ to sequence spaces over $\mathbb{Z}$, our $\mathcal{A}$-Hypercyclicity Criterion fails to deliver. In contrast, Bayart and Ruzsa have a characterization of frequently hypercyclic and upper frequently hypercyclic weighted shifts on $c_0(\mathbb{Z})$ and of frequently hypercyclic weighted shifts on $c_0$ (see \cite[Theorems 12, 13]{BaRu15}), while B\`es et al. characterize $\mathcal{A}$-hypercyclicity of weighted shifts on $c_0$ and $\ell_p$ also for not necessarily upper Furstenberg families (see \cite[Theorem 6]{BMPP15}). 

The dichotomy observed here reflects the difference between the two $\mathcal{A}$-Hypercyclicity Criteria that we discussed earlier, see Remark \ref{r-compAHC}.
\end{remark}

\section{Weighted shift operators and arbitrary Furstenberg families}\label{s-weiarb}

The last remark notwithstanding, we will in this section combine ideas from the proof of Theorem \ref{t-bahc} above with ideas of Bayart and Ruzsa \cite{BaRu15} to obtain a characterization of the $\mathcal{A}$-hypercyclicity of weighted shift operators for arbitrary Furstenberg families on a large class of sequence spaces. This will generalize \cite[Theorem 13]{BaRu15} and \cite[Theorem 6]{BMPP15}. The proof cannot use Baire category arguments and is therefore constructive. 

We begin again with the unweighted shift. In the following, if $X$ is a Fr\'echet sequence space then we let $(\|\cdot\|_n)_{n\geq 1}$ denote an increasing sequence of seminorms defining its topology.

\begin{theorem}\label{t-bahcgen}
Let $X$ be a Fr\'echet sequence space in which $(e_n)_{n\geq 0}$ is a basis. Suppose that the backward shift $B$ is an operator on $X$. Let $\mathcal{A}$ be an f.i.~Furstenberg family. 

\emph{(a)} If there exists a sequence $(\varepsilon_p)_{p\geq 1}$ of positive numbers with $\varepsilon_p\to 0$ as $p\to\infty$ and a sequence $(A_p)_{p\geq 1}$ of pairwise disjoint sets in $\mathcal{A}$ such that
\begin{itemize}
\item[\rm (i)] for any $p\geq 1$,
\[
\sum_{n\in A_p} e_{n+p} \quad \text{converges in $X$};
\]
\item[\rm (ii)] for any $p,q\geq 1$, any $m\in A_q$, and any $j=0,\ldots,p$,
\[
\Big\|\sum _{\substack{n\in A_p\\n>m}} e_{n-m+j}\Big\|_q <\min(\varepsilon_p,\varepsilon_q);
\]
\end{itemize}
then $B$ is $\mathcal{A}$-hypercyclic.

\emph{(b)} If $(e_n)_{n\geq 0}$ is an unconditional basis then the converse also holds.

Moreover, one may replace `there exists a sequence $(\varepsilon_p)_{p\geq 1}$' equivalently by `for any sequence $(\varepsilon_p)_{p\geq 1}$.'
\end{theorem}

\begin{proof} The final assertion of the theorem follows simply by passing to a subsequence of  $(A_p)_{p\geq 1}$ and noting the fact that $B$ is a continuous operator on $X$. Incidentally, the continuity of $B$ also implies the convergence of the series in (ii) under assumption (i).

(a) We first note that we have the following variant of property (F) in Theorem \ref{t-bahc}:  
\begin{itemize}
	\item [(F')] For any $p\geq 1$ there are a positive integer $n_p$ and a positive number $\alpha_p$ such that, whenever $x=(x_n)\in X$ and $\|x\|_{n_p}<\alpha_p$, then $|x_j|<1$ for $j=0,\ldots,p$. 
\end{itemize}
This implies that we may assume in addition to (i) and (ii) that
\begin{itemize}
\item[\rm (iii)] for any $p,q\geq 1$, $n\in A_p$, $m\in A_q$ with $n > m$, $n-m>q$.
\end{itemize}
Indeed, by passing again to a subsequence if necessary, we see by (ii), taking $j=0$, that we may even have that
for any $p,q\geq 1$ and for any $m\in A_q$
\[
\Big\|\sum _{\substack{n\in A_p\\n>m}} e_{n-m}\Big\|_{n_q} <\alpha_q,
\]
which by (F') implies (iii).

Now, by continuity of $B$ there are positive integers $k_p$ and positive numbers $\beta_p$, $p\geq 1$, such that if $\|x\|_{k_p}<\beta_p$ then
\begin{equation}\label{eq-bk}
\|B^{k} x\|_p <\frac{1}{p(p+1)2^p},\quad k=0,\ldots,p.
\end{equation}
Let $(A_p)_{p\geq 1}$ be a sequence in $\mathcal{A}$ that satisfies (i), (ii) and (iii) for the sequence
\begin{equation}\label{eq-epsk}
\varepsilon_p=\min\Big(\beta_p, \frac{1}{p(p+1)4^p}\Big),\quad p\geq 1.
\end{equation}
In view of (i) and the finite invariance of $\mathcal{A}$ we may assume that
\begin{itemize}
\item[\rm (i')] 
for any $p\geq 1$,
\[
\Big\|\sum_{n\in A_p} e_{n+p}\Big\|_{k_p} < \beta_p.
\]
\end{itemize}
In order to define an $\mathcal{A}$-hypercyclic vector, we fix a dense sequence $(y^{(p)})_{p\geq 1}$ of elements in $X$, where every $y^{(p)}$ is a finite sequence of the form
\[
y^{(p)} = \sum_{j=0}^p y^{(p)}_je_j,\quad \max_{0\leq j\leq p}|y_j^{(p)}|\leq p.
\]
We then define 
\[
x=\sum_{p=1}^\infty \sum_{n\in A_p} \sum_{j=0}^p y^{(p)}_j e_{n+j}.
\]
By (i') and \eqref{eq-bk} we have for $q\geq 1$ and $p\geq q$ that 
\begin{align*}
\Big\|\sum_{n\in A_p}\sum_{j=0}^p y^{(p)}_j e_{n+j}\Big\|_q&\leq \Big\|\sum_{n\in A_p}\sum_{j=0}^p y^{(p)}_j e_{n+j}\Big\|_p = \Big\|\sum_{j=0}^p y^{(p)}_j B^{p-j}\sum_{n\in A_p} e_{n+p}\Big\|_p\\
&\leq  \sum_{j=0}^p |y^{(p)}_j| \Big\|B^{p-j}\sum_{n\in A_p} e_{n+p}\Big\|_p\leq \frac{1}{2^p}.
\end{align*}
This shows that $x$ is a well-defined element of $X$.

It is now easy to see that $x$ is $\mathcal{A}$-hypercyclic for $B$. Indeed, for any $m\in A_q$,  $q\geq 1$, we have that
\[
B^mx-y^{(q)} = \sum_{p=1}^\infty \sum_{\substack{n\in A_p\\n\geq m}} \sum_{j=0}^p y^{(p)}_j e_{n-m+j}-\sum_{j=0}^q y^{(q)}_je_j= \sum_{p=1}^\infty \sum_{\substack{n\in A_p\\n> m}} \sum_{j=0}^p y^{(p)}_j e_{n-m+j}; 
\]
note that the terms $n< m$ disappear by (iii) because then $n-m+j<0$ if $j\leq p$. Moreover, (ii) and \eqref{eq-epsk} imply that for $j=0,\ldots,p$, if $p<q$ then
\[
 \Big\|\sum_{\substack{n\in A_p\\n> m}}  e_{n-m+j}\Big\|_q < \frac{1}{q(q+1)4^q}\leq \frac{1}{p(p+1)2^p2^q}, 
\]
while if $p\geq q$ then 
\[
 \Big\|\sum_{\substack{n\in A_p\\n> m}}  e_{n-m+j}\Big\|_q < \frac{1}{p(p+1)4^p}\leq\frac{1}{p(p+1)2^p2^q}.
\]
Putting everything together we obtain that, for any $m\in A_q$, $q\geq 1$,
\[
\|B^mx-y^{(q)}\|_q \leq  \sum_{p=1}^\infty \sum_{j=0}^p |y^{(p)}_j|\Big\|\sum_{\substack{n\in A_p\\n> m}}  e_{n-m+j}\Big\|_q\leq \sum_{p=1}^\infty\frac{1}{2^p2^q} = \frac{1}{2^q}.
\]
This confirms that $x$ is $\mathcal{A}$-hypercyclic for $B$.

(b) The proof is very similar to that of Theorem \ref{t-bahc}. We will need property (F') above and the following variant of (M) that can be proved in the same way:
\begin{itemize}
	\item[(M')] If $(x_n)\in X$ and $(a_n)$ is a bounded sequence of scalars then $(a_nx_n)\in X$. Moreover, for any $p\geq 1$ there are a positive integer $m_p$ and a positive number $C_p$ and such that, for any $x=(x_n)\in X$ and $(a_n)\in\ell_\infty$, $\|(a_nx_n)\|_p\leq  C_p \sup_n|a_n|\, \|x\|_{m_p}$. 
\end{itemize}
Now let $x\in X$ be an $\mathcal{A}$-hypercyclic vector for $B$, and let $(\varepsilon_p)_{p\geq 1}$ be a decreasing sequence of positive numbers that is bounded by 1. For $p\geq 1$, let $m_p$, $C_p\geq 1$, $n_p$, and $\alpha_p$ be given by properties (M') and (F'). We may assume that the numbers $C_p$ satisfy, in addition, that, if $q=p\geq 1$ with $0\leq j < k\leq p$ or if $q>p\geq 1$ with $0\leq j\leq p$ and $0\leq k\leq q$, then
\begin{equation}\label{eq-dpf}
(k+1)\frac{C_q}{\varepsilon_q}-(j+1)\frac{C_p}{\varepsilon_p}\geq 2.
\end{equation}

Let
\[
y^{(p)} = \sum_{j=0}^p\Big(1+(j+1)\frac{C_p}{\varepsilon_p}\Big) e_j,\quad p\geq 1,
\]
and choose positive integers $l_p$ and positive numbers $\rho_p$ such that the open balls of radius $\rho_p$ with respect to the seminorm $\|\cdot\|_{l_p}$ around $y^{(p)}$, $p\geq 1$, are pairwise disjoint. Then, by $\mathcal{A}$-hypercyclicity of $x$, there are sets $A_p\in\mathcal{A}$, $p\geq 1$, such that, for all $n\in A_p$, 
\[
\|B^nx-y^{(p)}\|_{\max(l_p,n_p,m_p)}<\min \Big(\rho_p,\alpha_{p},\frac{\varepsilon_p}{C_p}\Big).
\]
By the choice of the $l_p$ and $\rho_p$ the sets $A_p$, $p\geq 1$, are pairwise disjoint. We claim that they satisfy the conditions (i) and (ii). 

First, it follows from (F') that, for any $n\in A_p$ and $j=0,\ldots,p$,
\begin{equation}\label{eq-dp}
\Big|x_{n+j}-\Big(1+(j+1)\frac{C_p}{\varepsilon_p}\Big)\Big|<1,
\end{equation}
hence 
\[
\Big|\frac{1}{x_{n+j}}\Big|<\frac{\varepsilon_p}{(j+1)C_p}\leq \varepsilon_p\leq 1. 
\]
Thus by property (M') and the fact that $x\in X$ we obtain that
\[
\sum_{n\in A_p} \frac{1}{x_{n+p}}x_{n+p} e_{n+p}
\]
converges in $X$. This shows condition (i). 

Secondly, let $p,q\geq 1$, $m\in A_q$ and $j=0,\ldots,p$. If $p\leq q$ then 
\[
\|B^mx-y^{(q)}\|_{m_q} < \frac{\varepsilon_q}{C_q}= \frac{\min(\varepsilon_p,\varepsilon_q)}{C_q}\quad \text{and}\quad \sup_{n\in A_p}\Big|\frac{1}{x_{n+j}}\Big|\leq 1;
\]
and if $p> q$ then 
\[
\|B^mx-y^{(q)}\|_{m_q} < \frac{\varepsilon_q}{C_q}\leq \frac{1}{C_q} \quad \text{and}\quad \sup_{n\in A_p}\Big|\frac{1}{x_{n+j}}\Big|\leq \varepsilon_p = \min(\varepsilon_p,\varepsilon_q).
\]
In both cases, (M') implies that 
\begin{equation}\label{eq-M}
\Big\|\sum_{\substack{n\in A_p\\n>m}} \frac{1}{x_{n+j}}[B^mx-y^{(q)}]_{n-m+j} e_{n-m+j} \Big\|_q<\min(\varepsilon_p,\varepsilon_q);
\end{equation}
again, $[z]_k$ denotes the $k$th entry of the sequence $z$. But by \eqref{eq-dp}, we have for $n\in A_p$ with $n>m$ and $k=0,\ldots,q$ that
\[
\Big|x_{n+j} -x_{m+k}-(j+1)\frac{C_p}{\varepsilon_p}+(k+1)\frac{C_q}{\varepsilon_q}\Big|<2.
\]
This implies that $n+j\neq m+k$. Indeed, equality can hold by \eqref{eq-dpf} only if $p=q$ and $j=k$, which is impossible since $m>n$. Thus we have that $n-m+j>q$ and therefore
\[
[B^mx-y^{(q)}]_{n-m+j}=x_{n+j}.
\]
Hence \eqref{eq-M} implies that also condition (ii) is satisfied.
\end{proof}

In view of the results in \cite{BaRu15} and \cite{BMPP15} it is not surprising that our conditions involve a sequence  of sets $A_p$ that interact. It is not clear if, for general Furstenberg families, this is necessarily so. 

General weighted shifts can now be treated in the way described earlier.

\begin{theorem}\label{t-bwahcgen}
Let $X$ be a Fr\'echet sequence space in which $(e_n)_{n\geq 0}$ is a basis. Suppose that the weighted backward shift $B_w$ is an operator on $X$. Let $\mathcal{A}$ be an f.i.~Furstenberg family. 

\emph{(a)} If there exists a sequence $(\varepsilon_p)_{p\geq 1}$ of positive numbers with $\varepsilon_p\to 0$ as $p\to\infty$ and a sequence $(A_p)_{p\geq 1}$ of pairwise disjoint sets in $\mathcal{A}$ such that
\begin{itemize}
\item[\rm (i)] for any $p\geq 1$,
\[
\sum_{n\in A_p} \frac{1}{\prod_{\nu=1}^{n+p} w_\nu}e_{n+p} \quad \text{converges in $X$};
\]
\item[\rm (ii)] for any $p,q\geq 1$, any $m\in A_q$, and any $j=0,\ldots,p$,
\[
\Big\|\sum _{\substack{n\in A_p\\n>m}} \frac{1}{\prod_{\nu=1}^{n-m+j} w_\nu} e_{n-m+j}\Big\|_q <\min(\varepsilon_p,\varepsilon_q);
\]
\end{itemize}
then $B_w$ is $\mathcal{A}$-hypercyclic.

\emph{(b)} If $(e_n)_{n\geq 0}$ is an unconditional basis then the converse also holds.

Moreover, one may replace `there exists a sequence $(\varepsilon_p)_{p\geq 1}$' equivalently by `for any sequence $(\varepsilon_p)_{p\geq 1}$.'
\end{theorem}

Let us consider the special case of shifts on $c_0$, where we rephrase the conditions slightly in view of an application in the next section.

\begin{corollary}\label{c-bwahcgen}
Let $w=(w_n)_{n\geq 1}$ be a bounded sequence of non-zero scalars and $\mathcal{A}$ an f.i.~Furstenberg family. Then the weighted backward shift $B_w$ is $\mathcal{A}$-hypercyclic on $c_0$ if and only if there exists a sequence $(M_p)_{p\geq 1}$ of positive numbers with $M_p\to \infty$ as $p\to\infty$ and a sequence $(A_p)_{p\geq 1}$ of pairwise disjoint sets in $\mathcal{A}$ such that
\begin{itemize}
\item[\rm (i)] for any $p\geq 1$,
\[
|w_1w_2\cdots w_{n+p}| \to \infty\quad\text{as $n\to\infty$, $n\in A_p$};
\]
\item[\rm (ii)] for any $n\in A_p$, $m\in A_q$, $p\leq q$, with $n\neq m$,
\[
|w_1w_2\cdots w_{|n-m|+j}| > M_q,
\]
where $j=0,\ldots,p$ if $n>m$, and $j=0,\ldots,q$ if $n<m$.
\end{itemize}
\end{corollary}

\section{Examples}\label{s-ex}

In this section we simplify three known counter-examples. And we add a new one that sheds some light on Question \ref{q-chaos}.

Trivially, frequent hypercyclicity implies upper frequent hypercyclicity, which in turn implies reiterative hypercyclicity. While, in general, chaos and frequent hypercyclicity are incomparable concepts (see \cite{BaGr07}, \cite{Men15}), in the framework of weighted backward shifts on F-sequence spaces in which $(e_n)$ is an unconditional basis, chaos implies frequent hypercyclicity, see \cite[Corollary 9.14]{GrPe11}. 

It is known that none of these three implications can be reversed:
\[
\text{reiterative hc} \underset{\rm (a)}{\centernot \Longrightarrow} \text{upper frequent hc} \underset{\rm (b)}{\centernot \Longrightarrow} \text{frequent hc} \underset{\rm (c)}{\centernot \Longrightarrow} \text{chaos}
\]
In each case, the first counter-example was constructed using a weighted backward shift on $c_0$. More precisely, (a) is due to B\`es, Menet, Peris and Puig \cite[Theorem 7]{BMPP15}, (b) to Bayart and Ruzsa \cite[Theorem 5]{BaRu15}, and (c) to Bayart and Grivaux \cite[Corollary 5.2]{BaGr07}. However, the proofs take up a combined twelve printed pages. Based on our work we can come up with substantially simpler counter-examples for (a) and (b), while the ideas of these examples and the characterization of frequently hypercyclic weighted shifts on $c_0$ (see Bayart and Ruzsa \cite[Theorem 13]{ BaRu15} or Corollary \ref{c-bwahcgen} above) allow us to do likewise for (c). In fact, our counter-example for (c) can be strengthened to be even very frequently hypercyclic.

Before going into details we note that the characterizations of dynamical properties of the weighted shifts $B_w$ take a simpler form when using the products of the weights:
\[
\varpi_n=w_1w_2\cdots w_n,\quad n\geq 0.
\]
This is due to the conjugacy of $B_w$ with the unweighted shift $B$ on $c_0(v)$, where $v=(\frac{1}{\varpi_n})_{n\geq 0}$; see Section \ref{s-weiupp}. We can, and will, therefore simplify the notation by working in terms of $\varpi$ instead of $w$. Note that $B_w$ is an operator on $c_0$ if and only if $\sup_{n\geq 0}|\frac{\varpi_{n+1}}{\varpi_n}|<\infty$.

\begin{theorem}[B\`es, Menet, Peris, and Puig \cite{BMPP15}]  There exists a weighted backward shift $B_w$ on $c_0$ that is reiteratively hypercyclic but not upper frequently hypercyclic.
\end{theorem}

\begin{proof} Motivated by B\`es et al. we consider 
\[
S = \bigcup_{j,l\geq 1}S_{j,l}, \quad S_{j,l}=[l10^j-j,l10^j+j],
\]
where the intervals are sets of non-negative integers. We define
\[
\varpi_n = \max \{ 2^{\nu} : n\in S_{j,l}, n=l10^j-j+\nu, j,l\geq 1\}
\] 
if $n\in S$, and $\varpi_n = 1$ otherwise. It is easy to see that $\frac{\varpi_{n+1}}{\varpi_n}\leq 2$ for all $n\geq 0$, so that $B_w$ is an operator on $c_0$. Incidentally, the bases 2 and 10 are only chosen for the sake of better readability; any other integer bases larger than 1 work just as well.

We will first verify the conditions in Corollary \ref{c-bwahc} for the family $\mathcal{A}$ of sets of positive upper Banach density. 

Thus, let $p\geq 0$ and $M>0$. We fix $k\geq p$ such that $2^{k+p}>M$, and define
\[
A= \{10^{j_m}+l10^k : l\in [0,m], m\geq 1\},
\]
where $j_m$ is chosen such that $j_m\geq m10^k$, $m\geq 1$. Note that $\min A \geq p$.

First of all, since each interval $[10^{j_m},10^{j_m}+N10^k],$ $N\geq 0$, has cardinality $N10^k+1$ and contains at least $N+1$ elements of $A$ whenever $m\geq N$, we have that $\overline{\text{Bd}}\; A\geq 10^{-k}>0$.

Next, the fact that $m10^k\leq j_m$ for $m\geq 1$ ensures that 
\[
\varpi_n \geq 2^{j_m}
\]
whenever $n\in A$ is of the form $n=10^{j_m}+l10^k$, $l\in [0,m]$, $m\geq 1$. This shows that $\varpi_n \to \infty$ along $A$.

Finally, let $n,m\in A$, $n>m$. Then $n-m$ is a multiple of $10^k$ (note that any $j_m\geq k$), and since $p\leq k$ we obtain that
\[
\varpi_{n-m+p} \geq 2^{k+p}>M.
\]

Thus we have shown that $B$ is reiteratively hypercyclic.

Next suppose that $B$ is even upper frequently hypercyclic. Then by Corollary \ref{c-bwahc} there is a set $A$ of positive upper density such that
\[
\varpi_n\to \infty
\]
as $n\to\infty$ along $A$. Like B\`es et al. we consider the sets
\[
D_j = \{n\geq 1 : \varpi_n \geq 2^{j}\},\quad j\geq 1.
\]
Since each element of $A$ belongs to $D_j$ with at most finitely many exceptions, we have that $\overline{\text{dens}}\; A \leq \overline{\text{dens}}\; D_j$ for any $j\geq 1$. To come up with the desired contradiction it thus suffices to show that $\overline{\text{dens}}\; D_{2k+1}\to 0$ as $k\to\infty$.

Now, if $n\in S$, any interval $S_{j,l}$ that contains $n$ contributes at most $2^{2j}$ to the weight $\varpi_n$. Therefore, if 
\[
\varpi_n \geq 2^{2k+1},
\]
then $n$ must satisfy
\[
n\in S_{j,l}\quad \text{for some $j >k$, $l\geq 1$}.
\]

So, let $N\geq 0$. Then $S_{j,l}$ intersects $[0,N]$ only if
\[
l\leq \Big\lfloor \frac{N+j}{10^j}\Big\rfloor.
\]
Thus, for any $j\geq 1$, the interval $[0,N]$ meets at most $\lfloor\frac{N+j}{10^j}\rfloor$ intervals $S_{j,l}$, $l\geq 1$, each of which contains $2 j+1$ elements. Hence
\[
\text{card} (D_{2k+1}\cap [0,N]) \leq \sum_{j=k+1}^\infty (2j+1) \Big\lfloor    \frac{N+j}{10^j}\Big\rfloor \leq (N+1)\sum_{j=k+1}^\infty \frac{2j+1}{10^j} + \sum_{j=k+1}^\infty \frac{2j^2}{10^j},
\]
and therefore
\[
\overline{\text{dens}}\;D_{2k+1} \leq \sum_{j=k+1}^\infty \frac{2j+1}{10^j} \to0
\]
as $k\to\infty$. This had to be shown.
\end{proof} 

A straightforward modification of the above example allows us to provide a new, and much shorter, proof of the following.

\begin{theorem}[Bayart and Ruzsa \cite{BaRu15}]  There exists a weighted backward shift $B_w$ on $c_0$ that is upper frequently hypercyclic but not frequently hypercyclic.
\end{theorem}

\begin{proof} We redefine the sets $S_{j,l}$ of the previous proof by taking, for $j\geq 1$,
\[
S_{j,l}=\begin{cases}[10^j-10^{j-1},10^j+10^{j-1}], &\text{if } j\in Q, l= 1,\\ 
[10^j-j,10^j+j],&\text{if }j\notin Q, l= 1,\\
[l10^j-j,l10^j+j],&\text{if } l\geq 2,\end{cases}
\]
where $Q=\{q^2:q\geq 1\}$, and we set
\[
S=\bigcup_{j,l\geq 1} S_{j,l}.
\]
If $s_{j,l}$ denotes the least value in $S_{j,l}$ then we define
\[
\varpi_n = \max \{ 2^{\nu} : n\in S_{j,l}, n=s_{j,l}+\nu, j,l\geq 1\}
\] 
if $n\in S$, with $\varpi_n=1$ otherwise. As before, $B_w$ is an operator on $c_0$.

In order to prove the positive part of the assertion, let $p\geq 0$ and $M>0$. We fix $k\geq p$ such that $2^{k+p}>M$, and define
\[
A= \{10^{q^2}+l10^k : l\in [0,10^{q^2-1-k}], q\geq k+1\},
\]
which is contained in $[p,\infty)$. Since, for any $q\geq k+1$, the interval $[0, 10^{q^2}+10^{{q^2}-1}]$ contains at least $10^{{q^2}-1-k}$ elements of $A$, we have that
\[
\overline{\text{dens}}\; A\geq \lim_{q\to\infty} \frac{10^{{q^2}-1-k}}{10^{q^2}+10^{{q^2}-1}+1}= \frac{10^{-k}}{11}>0,
\]
so that $A$ has positive upper density. 

Now, if $n\in A$ is of the form $n=10^{q^2}+l10^k$, $l\in [0,10^{{q^2}-1-k}]$, $q\geq k+1$, then we have that $10^{q^2}\leq n\leq  10^{q^2}+10^{{q^2}-1}$, so that 
\[
\varpi_n \geq 2^{10^{q^2-1}},
\]
which shows that $\varpi_n \to \infty$ along $A$.

Moreover, let $n,m\in A$, $n>m$. Then $n-m$ is a multiple of $10^k$. Hence the fact that $p\leq k$ implies that
\[
\varpi_{n-m+p} \geq 2^{k+p}>M.
\]

Altogether we have shown by Corollary \ref{c-bwahc} that $B_w$ is upper frequently hypercyclic.

For the negative part of the assertion, suppose on the contrary that $B_w$ is even frequently hypercyclic. Then by \cite[Theorem 13]{BaRu15} or Corollary \ref{c-bwahcgen} above there is a set $A$ of positive lower density such that
\[
\varpi_n\to \infty
\]
as $n\to\infty$ along $A$. In order to obtain the desired contradiction we need only show that, as before, the sets
\[
E_j = \{n\geq 1 : \varpi_n \geq 2^{j}\},\quad j\geq 1,
\]
satisfy $\underline{\text{dens}}\; E_{2k+1}\to 0$ as $k\to\infty$. But since the set $S$ in the present proof is obtained from the corresponding set in the previous proof by adding the set
\[
R:= \bigcup_{q=1}^\infty S_{q^2,1}
\]
we have that
\[
E_j\subset D_j \cup R,\quad j\geq 1,
\]
where the $D_j$ are the sets defined in the previous proof. Now, for any $q\geq 2$, the interval $[0, 10^{q^2}-10^{q^2-1}-1]$ contains at most 
\[
\sum_{j=1}^{q-1}(2.10^{j^2-1}+1) \leq  3\sum_{j=0}^{(q-1)^2-1}10^j \leq \frac{1}{3} 10^{q^2-2q+1}
\]
elements of $R$, whence
\[
\underline{\text{dens}}\; R\leq \lim_{q\to\infty} \frac{10^{q^2-2q+1}}{3(10^{q^2}-10^{{q^2}-1})}= 0,
\]
so that $R$ has zero lower density. In the previous proof it has already been observed that, as $k\to\infty$, $\overline{\text{dens}}\; D_{2k+1}\to 0$, which implies that 
\[
\underline{\text{dens}}\; E_{2k+1}\to 0,
\]
as had to be shown.
\end{proof} 

We next give a new proof that frequent hypercyclicity does not imply chaos. For this we invoke Corollary \ref{c-bwahcgen}; moreover, motivated by the previous two examples, we use divisibility arguments to simplify the proof.

\begin{theorem}[Bayart and Grivaux \cite{BaGr07}]\label{t-BaGr}   
There exists a weighted backward shift $B_w$ on $c_0$ that is frequently hypercyclic but not chaotic.
\end{theorem}

\begin{proof} In fact, the weights are easily described. We fix $a_0=0$ and
\[
a_k = 5.10^{k^2}+5.10^{(k-1)^2}+\ldots+5.10^1,\quad k\geq 1,
\]
and we let 
\[
\varpi_n = 2^{\nu}
\]
whenever $n\in [a_{k-1},a_k)$, $k\geq 1$, with $\nu=n-a_{k-1}$. A crucial property of the $a_k$ is that if a multiple $l.10^{p^2}$ of $10^{p^2}$, $p\geq 1$, is smaller than (respectively at least) $a_k$ then so is $l.10^{p^2}+10^{(p-1)^2}$ (respectively $l.10^{p^2}-10^{(p-1)^2}$). Since $p\leq 10^{(p-1)^2}$ we have that, for any $l,p,k\geq 1$,
\begin{equation}\label{eq-ak}
a_{k-1} \leq l.10^{p^2} < a_k \quad \Longrightarrow \quad a_{k-1}+p\leq l.10^{p^2}< l.10^{p^2}+p < a_k.
\end{equation}

Again, $B_w$ is an operator on $c_0$, and it is not chaotic because $\varpi_n=1$ for $n=a_k$, $k\geq 0$, and hence $\varpi_n\not\to \infty$ as $n\to\infty$; see \cite[Example 4.9(a)]{GrPe11}.

We will now show that $B_w$ satisfies conditions (i) and (ii) of Corollary \ref{c-bwahcgen} for the Furstenberg family of sets of positive lower density with $M_p=2^{p-1}$, $p\geq 1$. In order to construct the desired sets $A_p$ we start with 
\[
C_p = \{ l.10^{p^2} : l\geq 1\},\quad p\geq 1, 
\]
which has (lower) density $10^{-p^2}$.

Now, if $n,m\in C_p$, $n\neq m$, then $|n-m|$ is a non-zero multiple of $10^{p^2}$, and we have that $a_{k-1}\leq |n-m| < a_k$ for some $k\geq 1$. Property \eqref{eq-ak} then implies that, for $j=0,\ldots,p$,
\[
a_{k-1}+p\leq |n-m|+j < a_k,
\]
so that
\[
\varpi_{|n-m|+j}\geq 2^{p}>M_p.
\]
Hence condition (ii) of Corollary \ref{c-bwahcgen} holds for the sets $C_p$ and $p=q$.

We will now thin out each set $C_p$ appropriately. First we set
\[
X =\bigcup_{k\geq 0} [a_k,a_k+k].
\]
Let $n\in C_p\setminus X$. Then 
\[
a_{k-1}\leq n<a_k
\]
implies by definition of $X$ and property \eqref{eq-ak} that we even have
\[
a_{k-1}+k-1 < n+p<a_k,
\]
so that
\[
\varpi_{n+p}\geq 2^{k}.
\]
Hence condition (i) of Corollary \ref{c-bwahcgen} holds for the sets $C_p\setminus X$. 

Next, for $q\geq 2$ let
\[
Y_q=\bigcup_{l\geq 1}\bigcup_{k\geq 0}\big( [l.10^{q^2}+a_k-q, l.10^{q^2}+a_k+q)\cup (l.10^{q^2}-a_k-q, l.10^{q^2}-a_k+q]\big).
\]
If $n\in C_p\setminus \bigcup_{q>p}Y_q$ and $m=l.10^{q^2}\in C_q$ with $p<q$ then we have for any $k\geq 0$
\[
|n-m|\notin [a_k-q,a_k+q),
\]
hence, for any $j=0,\ldots,q$,
\[
\varpi_{|n-m| +j}\geq 2^{q}>M_q.
\]
This shows that the sets $C_p\setminus \bigcup_{q>p}Y_q$ satisfy condition (ii) of Corollary \ref{c-bwahcgen} for $p<q$.

Altogether we have that the sets
\[
A_p:=C_p\setminus\Big(X\cup \bigcup_{q>p}Y_q\Big),\quad p\geq 1,
\]
satisfy the conditions of Corollary \ref{c-bwahcgen}; note that these sets are pairwise disjoint because $C_q\subset Y_q$, $q\geq 2$. It remains to show that each $A_p$ has positive lower density. 

Obviously, $X$ has density 0. For $Y_q$ we note that for $k\geq q$ we have that
\[
a_k = 5.10^{k^2} +\ldots + 5.10^{q^2} + 5.10^{(q-1)^2} +\ldots + 5.10^{1} = m.10^{q^2} + a_{q-1}
\]
for some $m\geq 1$, hence
\[
Y_q\cap\mathbb{N}_0\subset \bigcup_{l\geq 1}\bigcup_{k=0}^{q-1}\big( [l.10^{q^2}+a_k-q, l.10^{q^2}+a_k+q)\cup (l.10^{q^2}-a_k-q, l.10^{q^2}-a_k+q]\big).
\]
Now, let $N\geq 0$. Then an interval $[l.10^{q^2}+a_k-q,l.10^{q^2}+a_k+q)$ meets $[0,N]$ only if $l.10^{q^2}+a_k-q\leq N$, which implies that $l\leq \frac{N+q}{10^{q^2}}$. On the other hand, an interval $(l.10^{q^2}-a_k-q,l.10^{q^2}-a_k+q]$ meets $[0,N]$ only if $l.10^{q^2}-a_k-q+1\leq N$; since we need only consider $k\leq q-1$, this implies that $l\leq \frac{N+a_{q-1}+q}{10^{q^2}}$. Hence the number of elements from $\bigcup_{q>p}Y_q\cap\mathbb{N}_0$ in $[0,N]$ is at most
\[
\sum_{q=p+1}^\infty 2q\frac{2N+2q+a_{q-1}}{10^{q^2}};
\]
this series converges since $a_{q-1}\leq 60.10^{q^2-2q}$. Therefore the upper density of $\bigcup_{q>p}Y_q\cap \mathbb{N}_0$ does not exceed
\[
\sum_{q=p+1}^\infty \frac{4q}{10^{q^2}}\leq \sum_{\nu=1}^\infty \frac{1}{10^{p^2+\nu}} = \frac{1}{9}10^{-p^2}.
\]

Altogether we have that the set $A_p$ has lower density at least $\frac{8}{9}10^{-p^2}$, which completes the proof.
\end{proof}

We note that Bayart and Grivaux have obtained further properties of their weighted backward shift $B_w$ on $c_0$: It is frequently hypercyclic but does not have any unimodular eigenvalues, it is not mixing, and it does not admit any non-trivial invariant Gaussian measure (that is, apart from the Dirac measure at 0), see \cite[Section 5.7]{BaGr07}. However, this is true for any frequently hypercyclic, non-chaotic weighted backward shift $B_w$ on $c_0$. Indeed, the unimodular eigenvectors of $B_w$ are the non-zero multiples of $(\frac{\lambda^n}{w_1\cdots w_n})_{n\geq 0}$, $|\lambda|=1$, which only belong to $c_0$ if $B_w$ is chaotic; for the assertions on mixing and the Gaussian measures use \cite[Example 4.9(a)]{GrPe11} and the proof of \cite[Corollary 5.5]{BaGr07}, respectively. In particular, the weighted shift constructed above has these additional properties. 

We will finally show that a variant of this operator is even very frequently hypercyclic.

\begin{theorem} 
There exists a weighted backward shift $B_w$ on $c_0$ that is very frequently hypercyclic but not chaotic. In particular, it does not have any unimodular eigenvalues, it is not mixing, and it does not admit any non-trivial invariant Gaussian measure.
\end{theorem}

\begin{proof} We will modify the construction in the previous proof. We set $a_0=0$ and
\[
a_k = 5.10^{m_k}+5.10^{m_{k-1}}+\ldots+5.10^{m_1},\quad k\geq 1,
\]
where $(m_k)_{k\geq 1}$ is a strictly increasing sequence of positive integers such that $m_1=1$ and, for any $q\geq 2$,
\begin{equation}\label{eq-mk}
2(a_{q-1}+q)+3.10^{m_{q-1}}<10^{m_q}.
\end{equation}
As before, we consider the weighted backward shift $B_w$ on $c_0$ given by
\[
\varpi_n = 2^{\nu}
\]
whenever $n\in [a_{k-1},a_k)$, $k\geq 1$, with $\nu=n-a_{k-1}$. Then $B_w$ is not chaotic, and we need to show that it is very frequently hypercyclic for a suitable choice of $(m_k)$.

The starting point is provided again by the sets
\[
C_p = \{ l.10^{m_p} : l\geq 1\},\quad p\geq 1, 
\]
which we are going to thin out. This is done inductively. 

Let $q\geq 2$. For any choice of $d_q\in[0,10^{m_{q-1}})$ the set 
\[
Y_q:=\bigcup_{l\geq 1}\ [l.10^{m_q}-d_q-a_{q-1}-q, l.10^{m_q}+d_q+a_{q-1}+q]
\]
consists of pairwise disjoint intervals, between any two of which there is at least one element of $C_{q-1}$; this follows from \eqref{eq-mk}. We will use the freedom of the $d_q$ for the following. We first set $d_2=0$; now if $q\geq 3$ then the left-end points of the intervals in $Y_{q-1}$ are $10^{m_{q-1}}$ apart; thus, by choosing $d_q$ appropriately we can ensure that each left-end point of an interval in $Y_q$ coincides with the left-end point of an interval in $Y_{q-1}$. By symmetry the same is then true for right-end points. 

This finishes our construction. We now consider, as in the previous proof, the sets 
\[
X =\bigcup_{k\geq 0} [a_k,a_k+k]
\]
and, for $p\geq 1$,
\[
A_p:=C_p\setminus\Big(X\cup \bigcup_{q>p}Y_q\Big).
\]
The latter are pairwise disjoint because $C_q\subset Y_q$ for $q\geq 2$. It follows exactly as before that condition (i) and the case $p=q$ in condition (ii) of Corollary \ref{c-bwahcgen} are satisfied with $M_p=2^{p-1}$, $p\geq 1$. The case $p<q$ in condition (ii) follows in the same way after noting that, for $k\geq q$,
\[
a_k = m.10^{m_q}+a_{q-1}
\]
for some $m\geq 1$ and hence, for any $q\geq 2$,
\begin{align*}
Y_q &\supset \bigcup_{l\geq 1}\bigcup_{k= 0}^{q-1}\big( [l.10^{m_q}+a_k-q, l.10^{m_q}+a_k+q)\cup (l.10^{m_q}-a_k-q, l.10^{m_q}-a_k+q]\big)\\
&\supset \bigcup_{l\geq 1}\bigcup_{k\geq 0}\big( [l.10^{m_q}+a_k-q, l.10^{m_q}+a_k+q)\cup (l.10^{m_q}-a_k-q, l.10^{m_q}-a_k+q]\big) \cap \mathbb{N}_0.
\end{align*}

It remains to show that each set $A_p$ belongs to the Furstenberg family $\mathcal{A}_{\text{Hin}}$, see Example \ref{ex-fur1}(c), at least for a suitable choice of $(m_k)$. 

Since $X$ has density 0 it is easy to see that it suffices to show that, for any $p\geq 1$,
\[
B_p:=C_p\setminus\bigcup_{q>p}Y_q\in \mathcal{A}_{\text{Hin}}.
\]
To this end, fix $p\geq 1$. It follows from our construction that $\bigcup_{q>p} Y_q$ consists of pairwise disjoint (maximally chosen) intervals of varying sizes, each end-point of which is also an end-point of one of the intervals in $Y_{p+1}$. Hence between any two of these maximal intervals there is an element of $C_p$, so that to the right of any interval in $\bigcup_{q>p} Y_q$ there is an element from $C_p$ at distance at most $10^{m_p}$. As a consequence, taking the unions of sufficiently many sets $B_p-n$ will gradually fill up the holes caused by the $Y_p$ (and those between the elements of $C_p$). More precisely, fix $r\geq p+1$ and write
\[
B_p= C_p\setminus\bigcup_{q=p+1}^r Y_q\setminus\bigcup_{q=r+1}^\infty Y_q.
\]
Since the maximal length of an interval in $\bigcup_{q=p+1}^r Y_q$ is at most
\[
L_r=2(10^{m_{r-1}}+a_{r-1}+r)+1,
\]
we find with $N_r:=10^{m_p}+L_r$ that
\[
\bigcup_{n=0}^{N_r} (B_p-n) \supset \mathbb{N}_0\setminus \bigcup_{q=r+1}^\infty Y_q.
\]
In order to estimate the lower density of this set, let $M\geq 0$. Then an interval $[l.10^{m_q}-d_q-a_{q-1}-q, l.10^{m_q}+d_q+a_{q-1}+q]$ meets $[0,M]$ only if
\[
l\leq \frac{M+d_q+a_{q-1}+q}{10^{m_q}}\leq \frac{M+8.10^{m_{q-1}}}{10^{m_q}},
\]
and each such interval contains $2(d_q+a_{q-1}+q)+1\leq 17.10^{m_{q-1}}$ elements. Hence the number of elements of $\bigcup_{q=r+1}^\infty Y_q$ in $[0,M]$ is at most
\[
\sum_{q=r+1}^\infty 17.10^{m_{q-1}} \frac{M+8.10^{m_{q-1}}}{10^{m_q}}.
\]
If, finally, $(m_k)$ increases sufficiently fast, then this series converges, and the upper density of $\bigcup_{q=r+1}^\infty Y_q$ does not exceed
\[
17\sum_{q=r+1}^\infty \frac{10^{m_{q-1}}}{10^{m_q}},
\]
which tends to 0 as $r\to\infty$. For this choice of $(m_k)$ we then have that, for any $p\geq 1$,
\[
\lim_{r\to\infty} \underline{\text{dens}}\; \bigcup_{n=0}^{N_r} (B_p-n)= 1,
\]
as had to be shown.
\end{proof}

In fact, examples of very frequently hypercyclic, non-chaotic operators can already be found in the literature. In \cite{BaMa09}, Bayart and Matheron construct Kalisch-type operators on $C_0[0,2\pi]$ (Section 5.5.4) and on a separable Hilbert space (Theorem 6.41) that admit ergodic measures of full support and hence are very frequently hypercyclic (see Example \ref{ex-fur1}(c) above), but neither is chaotic. We note that the first operator has no unimodular eigenvalues while the second does. 

These examples shed some light on Question \ref{q-chaos} above. They show that not even for the Furstenberg family $\mathcal{A}=\mathcal{A}_{\text{Hin}}$ one has that $\mathcal{A}$-hypercyclicity implies chaos. One might therefore conjecture that the answer to Question \ref{q-chaos} is negative.

\section*{Acknowledgements} The authors are grateful to Étienne Matheron, Quentin Menet and Vassili Nestoridis for helpful discussions.

~\\[2mm]
\noindent
\parbox[t][3cm][l]{7cm}{\small
\noindent
Antonio Bonilla\\
Departamento de An\'alisis Matem\'atico\\
Universidad de La Laguna\\
C/Astrof\'{\i}sico Francisco S\'anchez, s/n\\
38721 La Laguna, Tenerife, Spain\\
E-mail: a.bonilla@ull.es}
\parbox[t][3cm][l]{8cm}{\small
Karl-G. Grosse-Erdmann\\
Département de Mathématique, Institut Complexys\\
Universit\'e de Mons\\
20 Place du Parc\\
7000 Mons, Belgium\\
E-mail: kg.grosse-erdmann@umons.ac.be}

\end{document}